\documentclass[12pt]{elsarticle}
\usepackage{amsfonts,amssymb,amsmath,mathrsfs}
\usepackage{lineno,hyperref}
\usepackage{color}

\makeatletter
\def\ps@pprintTitle{%
	\let\@oddhead\@empty
	\let\@evenhead\@empty
	\let\@oddfoot\@empty
	\let\@evenfoot\@oddfoot
}
\makeatother

\newtheorem{theorem}{Theorem}[section]

\newtheorem{example}[theorem]{Example}
\newtheorem{proposition}[theorem]{Proposition}
\newtheorem{remark}[theorem]{Remark}
\newtheorem{corollary}[theorem]{Corollary}
\newtheorem{definition}[theorem]{Definition}
\numberwithin{equation}{section}

\newenvironment{proof}[1][\noindent \textbf{Proof: }]{#1}{ \hfill $\square$ \vspace{2mm}}

\begin{document}

\begin{frontmatter}

\title{Global  hypoellipticity  for a class of overdetermined systems of pseudo-differential operators on the torus}


%author 1
\author[addressUFPR]{Cleber de Medeira \corref{correspondingauthor}}
\cortext[correspondingauthor]{Corresponding author}
\ead{clebermedeira@ufpr.br}

%author 2
\author[addressUFPR]{Fernando de \'Avila Silva}
\ead{fernando.avila@ufpr.br}

\address[addressUFPR]{Department of Mathematics, Federal University of Paran\'a, Caixa Postal 19081, \\ CEP 81531-980, Curitiba, Brazil}

\begin{abstract}
This article studies the  global hypoellipticity 
of a class of overdetermined systems of pseudo-differential operators defined on the torus. The main goal consists in establishing  connections between the global hypoellipticity of the system and the  global hypoellipticity of its normal form. It is proved that an  obstruc\-tion of number-theoretical nature appears as a necessary condition to the  global hypoellipticity. Conversely, the  sufficiency is approached ana\-lyzing  three types of hypotheses: a H\"{o}rmander condition, logarithmic growth and super-logarithmic growth. 
\end{abstract}


\end{frontmatter}

%\linenumbers

\tableofcontents

\section{Introduction}

This work deals with the global hypoellipticity of the following system of   pseudo-differential operators 
\begin{equation}\label{system}
L_j = D_{t_j} + c_j(t_j) P_j(D_x),\quad   j=1,\ldots,n,
\end{equation}
on the torus $\mathbb{T}^{n+N} \simeq (\mathbb{R}/2\pi\mathbb{Z})^{n+N}$, 
where  $c_j$ are smooth functions on $\mathbb{T}_{t_j}^{1}$, $ D_{t_j}={i}^{-1}{\partial_{t_j}}$ and $P_j(D_x)$ are pseudo-differential operators of order $m_j$
defined by 
\begin{equation}\label{pdo}
P_j(D_x)  u(x) = \sum_{\xi \in \mathbb{Z}^N}{e^{i x \cdot \xi} p_j(\xi) \widehat{u}(\xi)}.
\end{equation}

The expression $\widehat{u}(\xi) = (2\pi)^{-N} \int_{\mathbb{T}^N}{e^{- i x \cdot \xi} u(x) dx}$, $\xi \in \mathbb{Z}^N$, denotes the Fourier coefficients of $u$ and for each $j$ the toroidal symbol  $p_j(\xi)$ belongs to the class $S^{m_j}(\mathbb{Z}^N)$ in the sense of Definition 4.1.7 in \cite{RT3}. In particular, there exists  a positive constant $C_j$ such that 
\begin{equation}\label{bound-symb}
|p_j(\xi)| \leq C_j \|\xi\| ^{m_j}, \quad \forall \xi \in \mathbb{Z}^{N}.
\end{equation}

By $(t,x)=(t_1,\ldots,t_n,x_1,\ldots,x_N)$ we denote the periodic variables in $\mathbb{T}^{n+N}$, while their duals  are denoted by 
$(\tau, \xi) = (\tau_1, \ldots, \tau_n, \xi_1, \ldots, \xi_N)  \in \mathbb{Z}^{n+N}.$ The \textit{normal form} of  system \eqref{system} is the constant coefficient  system 
\begin{equation}\label{normal-form}
L_{j0} = D_{t_j} + c_{j0} P_j(D_x),\quad   j=1,\ldots,n,
\end{equation}
where $c_{j0}=(2\pi)^{-1}\int_{\mathbb{T}^1}c_j(s)ds$.

By considering the associated closed $1$-form $c(t)=\sum_{j=1}^{n}c_j(t_j)dt_j$ defined on $\mathbb{T}^n$, we have that  system \eqref{system} (more precisely $iL_j$, $j=1,\ldots, n$) gives rise to a complex of pseudo-differential operators (see Section I.3 in \cite{treves76}), 
which, at the first step, acts in the following way
$$
\mathbb{L}u=d_tu+c(t,D_x)\wedge u,
$$
where $ u\in C^{\infty}(\mathbb{T}^{n+N})$ (or $u\in \mathcal{D}'(\mathbb{T}^{n+N}))$, $d_t$ denotes the exterior derivative on $\mathbb{T}^n$ and $c(t,D_x)\wedge u=\sum_{j=1}^{n}i c_j(t_j)P_j(D_x)udt_j$. Analogously, we  may regard the normal form of $\mathbb{L}$  by defining 
$$
\mathbb{L}_0 u=d_tu+c_0(D_x)\wedge u,
$$
where $c_0(D_x) \wedge u=\sum_{j=1}^{n}i c_{j0}P_j(D_x)udt_j$. 
 
Therefore, the study of the global hypoellipticity of  system \eqref{system}, which is the interest of this article, becomes equivalent  to performing this study on the complex of pseudo-differential operators, at the first step, as described in \cite{Maire80}, namely:

\begin{definition}\label{def-GH}
The operator $\mathbb{L}$ is said to be globally hypoelliptic if the conditions $u \in \mathcal{D}'(\mathbb{T}^{n+N})$  and $\mathbb{L}u\in\bigwedge^1 C^{\infty}(\mathbb{T}^{n+N})$ imply $u  \in C^{\infty}(\mathbb{T}^{n+N})$. Equi\-va\-lently,   system \eqref{system} is globally hypoelliptic when the conditions $u \in \mathcal{D}'(\mathbb{T}^{n+N})$ and   $L_j u \in C^{\infty}(\mathbb{T}^{n+N})$ for all $j\in\{1, \ldots, n\}$, imply $u  \in C^{\infty}(\mathbb{T}^{n+N})$.
\end{definition}

By an abuse of notation,  from now on we shall denote  system \eqref{system} by $\mathbb{L}$ and its normal form \eqref{normal-form} by 
$\mathbb{L}_0$.

We point out that the local hypoellipticity and local solvability of systems of type \eqref{system} are described  in the seminal work of Treves \cite{treves76}. On the other hand, 
the study of the global hypoellipticity of systems of differential operators on compact manifolds, with special attention to vector fields,  has been intensively studied. We highlight  Bergamasco, Cordaro and Malagutti's work \cite{BCM93}. These problems are also considered in terms of the scales of analytic and Gevrey spaces; see for instance  \cite{BERG99,AKM} and references therein.

Although the analysis on global hypoellipticity and global solvability  of systems of differential operators, defined on compact manifolds, composes a widely explored problem in the literature, less common is the study of such global properties dealing with classes of pseudo-differential operators. In this manner, this article contributes by presenting an approach for the study of the global hypoellipticity of unexplored classes of systems on the torus.

This article is organized as follows: Section \ref{Systems with constant coefficients} discusses systems with cons\-tant coefficients. There, we present a characterization of the global hypoellipticity  for that case (see Theorem \ref{general-constant-system}) by following the approach introduced by Greenfield and Wallach in \cite{GW1}.

Additionally, Subsection \ref{homo-symbols} discusses systems on $\mathbb{T}^{n+1}$ with homogeneous symbols. There, we  present connections between global hypoellipticity and certain \textit{simultaneous Diophantine approximations} (see Definition \ref{def-SDA} and Theorem \ref{homogeneous hypoellipticity}). 

In Section \ref{Systems with variable coefficients}, the case of variable coefficients is studied. Subsection \ref{sec necessary conditions} is directed to the study of necessary conditions to the global hypoellipticity of  system \eqref{system}, while  Subsection \ref{sec-suff-cond} is dedicated to the study of sufficient conditions. 

In Theorem \ref{necessary condition} we show that the global hypoellipticity of the normal form \eqref{normal-form} is a necessary condition to the global hypoellipticity of system \eqref{system}. On the other hand, in Theorem \eqref{main-suff}, we show that the opposite direction holds true, provided that some operator  $L_{j0}$ is globally hypoelliptic and that suitable hypotheses are considered on the functions
$$
\zeta \in \mathbb{T}^1  \mapsto \mathcal{M}_j(\zeta, \xi) \doteq c_j(\zeta)p_j(\xi), \quad\xi \in \mathbb{Z}^N.
$$
Several examples are presented in order to illustrate such hypotheses.

Finally, in Section \ref{sec-reduction}, the reduction to the normal form is discussed. Under appropriated conditions on the functions $\Im\mathcal{M}_j(\cdot, \xi)$, we will exhibit, in Theorem \ref{t-general-reduction},  an isomorphism  $\Psi$ defined on both  spaces $\mathcal{D}'(\mathbb{T}^{n+N})$ and $C^{\infty}(\mathbb{T}^{n+N})$ satisfying  the conjugation
\begin{equation*}
L_j = \Psi \circ L_{j0} \circ  \Psi^{-1}, \quad j=1, \ldots, n.
\end{equation*}
In particular, it follows from this conjugation that  $\mathbb{L}$ is globally hypoelliptic if and only if the same occurs with $\mathbb{L}_0$.

%================================================================
\section{Systems with constant coefficients}\label{Systems with constant coefficients}
%================================================================

In this section we investigate the global hypoellipticity of the following system with constant coefficient operators
\begin{equation}\label{constant-system}
L_j = D_{t_j} + P_j(D_x), \quad j=1,\ldots,n,
\end{equation}
defined on $\mathbb{T}_t^n\times\mathbb{T}_x^N$, where $P_j(D_x)$ are pseudo-differential operators of order $m_j$  with toroidal symbol $p_j(\xi)\in S^{m_j}(\mathbb{Z}^N)$.

The symbol of each $L_j$ is denoted by $\widehat{L}_j(\tau_j,\xi) = \tau_j + p_j(\xi)$ and we refer to
$$
\widehat{L}(\tau, \xi) = (\widehat{L}_1(\tau_1,\xi), \ldots, \widehat{L}_n(\tau_n,\xi)), \quad (\tau, \xi) \in \mathbb{Z}^n \times \mathbb{Z}^N,
$$
as the symbol of system \eqref{constant-system}.

In our approach  we make use of the following standard result: 
\begin{proposition}\label{prop-smooth-const}
	Let $\{a_{\eta}\}_{\eta \in \mathbb{Z}^{d}}$ be a sequence of complex numbers. The series given by 
	\begin{equation*}
	v(y) =  \sum_{\eta \in \mathbb{Z}^{d}}{a_{\eta} e^{i  y \cdot \eta}}
	\end{equation*}
	 converges  in $\mathcal{D}'(\mathbb{T}^{d})$ if and only if there  exist positive constants $C$, $M$ and $R$ such that
	\begin{equation}\label{smooth-coef-full}
	|a_{\eta}| \leq C \|\eta\|^{-M}, \quad \|\eta\|\geq R.
	\end{equation}
	Moreover, $v(\cdot) \in C^{\infty}(\mathbb{T}^{d})$ if and only if  \eqref{smooth-coef-full} is fulfilled for every $M>0$. In both cases we have $\widehat{v}(\eta) = a_{\eta}$.
\end{proposition}

We prove that  the global hypoellipticity of  system \eqref{constant-system} can be cha\-racte\-rized
by extending  the results of Greenfield's and Wallach's work \cite{GW1}, as stated below.

\begin{theorem}\label{general-constant-system}
System \eqref{constant-system} is globally hypoelliptic if and only if there exist positive constants $C,M$ and $R$ such that
\begin{equation}\label{GW}
\|\widehat{L}(\tau, \xi)\| \geq C \|(\tau, \xi)\|^{-M}, \quad  \|(\tau, \xi)\| \geq R.
\end{equation}
\end{theorem}

\begin{proof}
Suppose that \eqref{GW} fails. Then, we take an increasing sequence $\{(\tau_{\ell}, \xi_{\ell})\}_{\ell \in \mathbb{N}} \subset \mathbb{Z}^{n+N}$ such that
$$
\|\widehat{L}(\tau_{\ell},\xi_{\ell})\|< \|(\tau_{\ell},\xi_{\ell})\|^{-\ell}, \quad \forall \ell \in \mathbb{N}.
$$	

In view of Proposition \ref{prop-smooth-const}, we have that 
$
u(t,x)=\sum_{\ell=1}^{\infty} e^{i (t,x)\cdot (\tau_{\ell},\xi_{\ell})}
$ belongs to $\mathcal{D}'(\mathbb{T}^{n+N})\setminus C^{\infty}(\mathbb{T}^{n+N})$. For each $j \in \{1, \ldots, n\}$ we set  $f_j \doteq L_j u$. Then, 
$$
|\widehat{f}_j(\tau_{\ell},\xi_{\ell})|=|\widehat{L}_j(\tau_{j\ell},\xi_{\ell}) | \leq \|\widehat{L}(\tau_{\ell},\xi_{\ell})\|\leq \|(\tau_{\ell},\xi_{\ell})\|^{-\ell}, \quad \forall \ell \in \mathbb{N},
$$
hence, $f_j \in C^{\infty}(\mathbb{T}^{n+N})$ and the system is not globally hypoelliptic. \\

Conversely, assume that \eqref{GW} holds true and let $u\in\mathcal{D}'(\mathbb{T}^{n+N})$ be a solution of the equations
$$
L_j u  = f_j \in  C^{\infty}(\mathbb{T}^{n+N}), \quad j=1, \ldots, n.
$$

By replacing the formal Fourier series of $u$ and $f_j$ in the previous equations we obtain
\begin{equation}\label{eqq1}
|\widehat{L}_j(\tau_j,\xi)| |\widehat{u}(\tau,\xi)| = |\widehat{f}_j(\tau,\xi)|, \quad \forall (\tau,\xi) \in \mathbb{Z}^{n+N}.
\end{equation}

It follows from Proposition \ref{prop-smooth-const} that given $\nu>0$ there exist positive constants $C_\nu$ and $R_\nu$ such that 
\begin{equation}\label{eqq2}
\max_{j=1,\ldots,n} |\widehat{f}_j(\tau,\xi))|\leq C_\nu  \|(\tau,\xi)\|^{-(\nu+M)}, \quad \|(\tau,\xi)\|\geq R_\nu. 
\end{equation}

Therefore, by combining \eqref{GW}, \eqref{eqq1} and \eqref{eqq2} we obtain
$$
|\widehat{u}(\tau,\xi)| \leq C_{\nu} C^{-1} \|(\tau,\xi)\|^{-\nu}, \quad \|(\tau,\xi)\| \geq \widetilde{R},
$$
where $\widetilde{R} = \max\{R, R_{\nu}\}$. Hence, $u$ is a smooth function on $\mathbb{T}^{n+N}$ and the proof is complete.

\end{proof}

\begin{corollary}\label{coro-const-system}
If there exists $j\in \{1, \ldots, n\}$ such that $L_j$ is globally hypoelliptic on $\mathbb{T}_{t_j}^{1}\times \mathbb{T}_x^N$, then the system \eqref{constant-system} is globally hypoelliptic.	
\end{corollary}

\begin{proof}
If $L_j$ is globally hypoelliptic, then there exist positive constants
$C_j$, $M_j$ and $R_j$ such that 
\begin{equation}\label{GW-j}
|\widehat{L}_j(\tau_j, \xi)| \geq C_j \|(\tau_j, \xi)\|^{-M_j}, \quad 
\|(\tau_j, \xi)\| \geq R_j.
\end{equation}
Hence, \eqref{GW} is obtained by combining  \eqref{GW-j} and the following  inequality
$$
|\widehat{L}_j(\tau_j, \xi)| \leq  \|\widehat{L}(\tau, \xi)\|, \quad (\tau, \xi) \in \mathbb{Z}^{n+N}.
$$
\end{proof}

\begin{corollary}\label{coro-W-finito-const}
Consider the sets $\mathcal{Z}_j = \{ \xi \in \mathbb{Z}^N; \, p_j(\xi) \in \mathbb{Z}\}$, $j=1, \ldots, n$. If  system \eqref{constant-system} is globally hypoelliptic, then $\mathcal{Z} = \bigcap_{j=1}^{n} \mathcal{Z}_j$ is finite. 
\end{corollary}

\begin{proof}
If $\mathcal{Z}$ is  infinite, we may write $\mathcal{Z} = \{\xi_{\ell}\}_{\ell \in \mathbb{N}}$, where $\xi_{\ell}$ is an increasing sequence. Thus, let $\tau_{j\ell} \doteq - p_j(\xi_{\ell})$, for all $\ell \in \mathbb{N}$. It follows that
$$
\widehat{L}_j(\tau_{j\ell}, \xi_{\ell}) = \tau_{\ell} + p_j(\xi_{\ell}) = 0, \quad \ell \in \mathbb{N},
$$
for each  $j \in \{1, \ldots, n\}$, then $\widehat{L}(\tau_{1\ell}, \ldots, \tau_{n\ell}, \xi_{\ell}) = 0$, for all $\ell$ and, consequently, \eqref{GW} fails. 
\end{proof}

Proposition 3.3 in \cite{AGKM} proves that set $\mathcal{Z}=\{ \xi \in \mathbb{Z}^N; \, p(\xi) \in \mathbb{Z}\}$ being finite is a necessary condition to the global hypoellipticity of a single operator $L=D_t+P(D_x)$ on $\mathbb{T}^{1+N}$. However, the following example not only shows that the reciprocal of Corollary \ref{coro-const-system} is not true, but also exhibits a globally hypoelliptic system, where each $\mathcal{Z}_j$ in Corollary \ref{coro-W-finito-const} is infinite. 

\begin{example}\label{example const coeff}
	Consider the system $L_j = D_{t_j} + P_j(D_x)$,  $j=1,2$, where the symbols $p_j(\xi)$ are given as follows: 
	\begin{equation*}
	p_1(\xi)  =
	\left \{
	\begin{array}{l}
	\xi, \ \textrm{ if } \ \xi \ \textrm{ is odd} \\
	ia, \ \textrm{ if } \ \xi \ \textrm{ is even} 
	\end{array}
	\right.	
   \ \textrm{ and } \
   	p_2(\xi)  =
   \left \{
   \begin{array}{l}
   ib, \ \textrm{ if } \ \xi \ \textrm{ is odd} \\
   \xi, \ \textrm{ if } \ \xi \ \textrm{ is even}
   \end{array},
   \right.
	\end{equation*}
	where $a$ and $b$ are non-zero real constants.
	
	Hence,  
	$$
	\mathcal{Z}_1 = \{\xi \in \mathbb{Z}; \; \xi \ \textrm{is odd}\}
	\ \textrm{ and } \
	\mathcal{Z}_2 = \{\xi \in \mathbb{Z};\; \xi \ \textrm{is even}\},
	$$
	which implies $\mathcal{Z}_1\cap \mathcal{Z}_2 = \emptyset$ and 
	\begin{equation*}
	\widehat{L}(\tau_1,\tau_2,\xi) = 
	\left \{
	\begin{array}{l}
	(\tau_1 + \xi , \tau_2 + ib), \ \textrm{ if } \ \xi \ \textrm{ is odd}, \\
	(\tau_1 + ia , \tau_2 + \xi), \ \textrm{ if } \ \xi \ \textrm{ is even}. 
	\end{array}
	\right.
	\end{equation*}
	
	It follows that $\Im \widehat{L}(\tau_1,\tau_2,\xi) \neq 0$  for all $\xi \in \mathbb{Z}$, thus we may obtain a positive constant $C$ such that $\|\widehat{L}(\tau_1,\tau_2,\xi)\| \geq C$, for all $(\tau_1, \tau_2, \xi) \in \mathbb{Z}^3$. Thus, the system is globally hypoelliptic although each operator $L_j$ is not.
\end{example}

\begin{remark}\label{remark-berg}
The existence of globally hypoelliptic systems, where each ope\-ra\-tor $L_j$ is not globally hypoelliptic, was first	introduced by Bergamasco  in \cite{BERG99}. Precisely, from Example 4.9 in \cite{BERG99}, it follows  that there exist two Liouville numbers $\alpha$ and $\beta$ such that the system
$$
\mathbb{L} =
\left \{
\begin{array}{l}
L_1 = D_{t_1} - \alpha D_x, \\
L_2 = D_{t_2} - \beta D_x,
\end{array}
\right.
$$
is globally hypoelliptic on $\mathbb{T}^3$.
\end{remark}

%============================================================================
\subsection{Homogeneous symbols \label{homo-symbols}}
%============================================================================

By taking inspiration from work \cite{AGKM}, the main goal of this section is to investigate the global hypoellipticity of a special class of constant coefficient  systems on $\mathbb{T}^n_t \times \mathbb{T}_x^{1}$, where each  operator $P_j(D_x)$, defined on $\mathbb{T}^1$, has a homogeneous symbol of positive and rational order  $\kappa$, namely, 

%{\color{red} Acho que assim fica direto e n\~ao precisa chamar aten\c{c}\~ao par ao caso $\kappa<0$...}

$$
p_j(\xi)=|\xi|^{\kappa}p_j(\pm 1), \quad \forall \xi\in \mathbb{Z}_*.
$$

We will write
$$
p_j(1) = \alpha_j + i \beta_j \quad \textrm{and} \quad p_j(-1) = \widetilde{\alpha}_j + i \widetilde{\beta}_j,
$$ 
and $\alpha=(\alpha_{1},\ldots,\alpha_n)$, $\widetilde{\alpha}=(\widetilde{\alpha}_{1},\ldots,\widetilde{\alpha}_n)$, $\beta=(\beta_{1},\ldots,\beta_n)$ and $\widetilde{\beta}=(\widetilde{\beta}_{1},\ldots,\widetilde{\beta}_n)$. Therefore, the symbol of each operator $L_j$ in system \eqref{constant-system} can be written as follows 
\begin{equation}\label{eqq}
|\widehat{L}_j(\tau_j, \xi)| = |\tau_j+p_j(\xi)| = |\xi|^{\kappa} \cdot 
\left \{
\begin{array}{l}
\left| \dfrac{\tau_j}{|\xi|^{\kappa}} + (\alpha_j  + i \beta_j ) \right|,  \ \textrm{ if } \xi>0, \vspace{2mm} \\
\left| \dfrac{\tau_j}{|\xi|^{\kappa}} + (\widetilde{\alpha}_j   + i \widetilde{\beta}_j ) \right|, \ \textrm{ if } \ \xi<0.
\end{array}
\right.
\end{equation}

The idea here is to shed light on the  connections between global hypoellipticity and certain  simultaneous Diophantine approximations of real numbers by rationals with a common  denominator. To be more precise, we introduce the following definition.
\begin{definition}\label{def-SDA}	
	Let $\alpha=(\alpha_1,\ldots,\alpha_n)\in\mathbb{R}^n\setminus \mathbb{Q}^n$ be a vector and $\eta\in\mathbb{N}$. We say that $\alpha$ satisfies the condition $(SDA)_{\eta}$ if there exist a positive  constant $C$ and a sequence $(p_\ell,q_\ell)=(p_{1\ell},\ldots,p_{n\ell},q_\ell)\in\mathbb{Z}^n\times \mathbb{N}$ such that 
	\begin{equation}
	\max_{j=1,\dots,n}\left|\alpha_{j}^\eta-\dfrac{(p_{j\ell})^\eta}{q_\ell}\right|<  C q_{\ell}^{-\ell}, \quad \ell \in\mathbb{N}.
	\end{equation}
\end{definition}

It is important to point out that  several papers dealing with global properties such as solvability and hypoellipticity, for differential systems, are inte\-res\-ted in simultaneous Diophantine approximations, as the reader can see in  \cite{BERG99,BCM93,BdMZ12,BP,himonas99,Him01,moser}. The usual approach in these works is to consider Liouville vectors:
\begin{definition}[Liouville vector]
	We say that a vector  $\alpha=(\alpha_1,\ldots,\alpha_n)$ in $\mathbb{R}^n\setminus \mathbb{Q}^n$ is  Liouville if  there exist a positive  constant $C$ and a sequence $(p_\ell,q_\ell)=(p_{1\ell},\ldots,p_{n\ell},q_\ell)\in\mathbb{Z}^n\times \mathbb{N}$ such that
	
	\begin{equation}\label{Liouville vector}
	\max_{j=1,\dots,n}\left|\alpha_{j}-\dfrac{p_{j\ell}}{q_\ell}\right|<  C q_{\ell}^{-\ell}, \quad \ell \in\mathbb{N}.
	\end{equation}	
\end{definition}

We claim that a Liouville vector $\alpha\in\mathbb{R}^n\setminus \mathbb{Q}^n $ satisfies  $(SDA)_{\eta}$, for every $\eta\in\mathbb{N}$. Indeed, if $\eta=1$, it is obvious. In the other case, given a Liouville vector $\alpha\in\mathbb{R}^n\setminus \mathbb{Q}^n$, we may find  a positive  constant $C$ and a sequence $(p_\ell,q_\ell)=(p_{1\ell},\ldots,p_{n\ell},q_\ell)\in\mathbb{Z}^n\times \mathbb{N}$ satisfying
$$
\max_{j=1,\dots,n} \left|\alpha_j-\dfrac{p_{j\ell}}{q_\ell}\right|
\leq  C (q_\ell^{\eta} )^{-\ell}, \quad \ell \in\mathbb{N}.
$$

Thus, for each $j\in\{1,\ldots,n\}$ we have 
\begin{equation*}
\left|\alpha_j^{\eta}-\dfrac{(p_{j\ell})^{\eta}}{q_\ell^{\eta}}\right|= \left|\alpha_j-\dfrac{p_{j\ell}}{q_\ell}\right|\cdot \left|\sum_{k=1}^{\eta}\left(\dfrac{p_{j\ell}}{q_{\ell}}\right)^{\eta-k}\alpha_j^{k-1}\right|.
\end{equation*}

Since the sum in the right hand side of the above equation converges to $\eta\alpha_j^{\eta-1}$, there exists a positive constant $\widetilde{C}$ such that 
\begin{eqnarray*}
	\max_{j=1,\dots,n} \left|\alpha_j^{\eta}-\dfrac{(p_{j\ell})^{\eta}}{q_\ell^{\eta}}\right| \leq \widetilde{C} \max_{j=1,\dots,n} \left|\alpha_j-\dfrac{p_{j\ell}}{q_\ell}\right|
	\leq \widetilde{C} C (q_\ell^{\eta} )^{-\ell}, \quad \ell \in\mathbb{N}.
\end{eqnarray*}

The reciprocal of the previous claim is not true and the next example is an illustration. 

\begin{example}\label{SDA example}
	Let $\mathscr{L}=\sum_{k=1}^{\infty}10^{-k!}$ be  the Liouville constant. Given an integer $\eta\geq 2$, we have that $3/2 \mathscr{L}^\eta$ is a Liouville number, while $\sqrt[\eta]{3/2}\mathscr{L}$ is not (see \cite{Marques}). Thus, there exists  a sequence $(p_\ell,q_\ell)\in\mathbb{N}^2$, $q_\ell\rightarrow +\infty$,  such that 
	$$
	\left|\dfrac{3}{2}\mathscr{L}^{\eta}-\dfrac{p_\ell}{q_\ell}\right|< q_\ell^{-2\eta \ell}< q_{\ell}^{-\ell}, \quad \ell \in\mathbb{N}.
	$$
	Hence, $|3/2\mathscr{L}^{\eta} q_\ell-p_\ell|<q_{\ell}^{1-\ell}$ which implies 
	$p_\ell<q_\ell(3/2\mathscr{L}^{\eta}+q_\ell^{-\ell})$,
	for all $\ell \in\mathbb{N}$. Therefore, 
	$$
	2^{\eta}p_\ell^{\eta-1}q_\ell< 2^{\eta}(3/2\mathscr{L}^{\eta}+q_{\ell}^{-\ell})^{\eta-1}q_\ell^{\eta}\leq q_\ell^{\eta} q_\ell^{\eta}=q_\ell^{2\eta},
	$$ for $\ell$ sufficiently large.
	
	Now, consider the following  non Liouville vector $\alpha=(\sqrt[\eta]{3/2}\mathscr{L},\sqrt[\eta]{3\cdot2^{\eta-1}}\mathscr{L})$ in $\mathbb{R}^2\setminus\mathbb{Q}^2$. 
	Thus, we obtain 
	$$
	\left|\dfrac{3}{2}\mathscr{L}^{\eta} -\dfrac{p_\ell}{q_\ell}\right|< q_\ell^{-2\eta\ell}\; \text{and}\; \left|3\cdot2^{\eta-1}\mathscr{L}^{\eta} -\dfrac{2^{\eta}p_\ell}{q_\ell}\right|< 2^{\eta} q_\ell^{-2\eta\ell},
	$$
	or equivalently
	$$
	\left|\dfrac{3}{2}\mathscr{L}^{\eta} -\dfrac{2^{\eta} p_\ell^{\eta}}{2^{\eta}p_\ell^{\eta-1} q_\ell}\right|< q_\ell^{-2\eta\ell}\;\text{and}\; \left|3\cdot2^{\eta-1}\mathscr{L}^{\eta} -\dfrac{2^{2\eta}p_\ell^{\eta}}{2^{\eta}p_\ell^{\eta-1} q_\ell}\right|< 2^{\eta} q_\ell^{-2\eta\ell}.
	$$
	
	Finally, the previous argumentation implies  
	$$
	\max \left\{ \left|\dfrac{3}{2}\mathscr{L}^{\eta} -\dfrac{(2 p_\ell)^{\eta}}{2^{\eta}p_\ell^{\eta-1} q_\ell}\right|,
	\left|3\cdot2^{\eta-1}\mathscr{L}^{\eta} -\dfrac{(4p_\ell)^{\eta}}{2^{\eta}p_\ell^{\eta-1} q_\ell}\right| \right\}
	< 2^{\eta} (2^{\eta}p_\ell^{\eta-1} q_\ell)^{-\ell},
	$$
	for $\ell$ large enough. We conclude  that $\alpha$ satisfies $(SDA)_\eta$.
\end{example}

\begin{theorem}\label{homogeneous hypoellipticity}
	Admit that each symbol   $p_j=p_j(\xi)$ is positively homogeneous  of order $\kappa=\rho/\eta$, where  $\rho,\eta\in\mathbb{N}$ and $\gcd(\rho,\eta)=1$. Then, the system 
	$$ 
	L_j=D_{t_j}+P_j(D_x),\quad j=1,\ldots,n,
	$$
	defined on $\mathbb{T}_t^n\times \mathbb{T}_x^1$
	is globally hypoelliptic if and only if  $\alpha\in\mathbb{R}^n$ is neither rational nor satisfies $(SDA)_{\eta}$, whenever $\beta=0$. 
	Additionally, the same is true for $\widetilde{\alpha}$ and $\widetilde{\beta}$.	
\end{theorem}

\begin{proof} 
	If $\beta\neq 0$ and $\widetilde\beta\neq0$, then  condition \eqref{GW} is trivially obtained. Thus, we start by considering the case $\beta=0 $ and $\widetilde\beta\neq0$. 
	
	If $\alpha\in\mathbb{Q}^n$, by taking a positive integer $q$ such that $q \alpha\in\mathbb{Z}^n$ we have  that $(q \xi)^\eta\in \mathcal{Z}=\bigcap_{j=1}^n \mathcal{Z}_j$, for all $\xi\in \mathbb{N}$. Then, it follows from Corollary \ref{coro-W-finito-const} that the system is not globally hypoelliptic.
	
	If $\alpha\in\mathbb{R}^n\setminus\mathbb{Q}^n$, by combining Greenfield's and Wallach's condition  \eqref{GW} and \eqref{eqq}, we may consider the following equivalence:  the system is globally hypoelliptic if and only if  there exist positive constants $C,$ $M$ and $R$ such that
	\begin{equation}\label{eqqq1}
	\max_{j=1,\ldots,n}\left|\dfrac{\tau_j}{\xi^{\kappa}}+\alpha_j\right|\geq C(|\tau|+\xi)^{-M},
	\end{equation}
	for all $(\tau,\xi)=(\tau_1,\ldots,\tau_n,\xi)\in\mathbb{Z}^n\times\mathbb{N}$ such that $|\tau|+\xi\geq R.$

	Note that we can write 
	\begin{equation}\label{nice equality}
	\left|\dfrac{(\tau_j)^\eta}{\xi^{\rho}}+\alpha_j^\eta\right|= \left|\dfrac{\tau_j}{\xi^{\kappa}}+\alpha_j\right|\cdot \left|\sum_{m=1}^{\eta}\left(\dfrac{\tau_j}{\xi^{\kappa}}\right)^{\eta-m}\alpha_j^{m-1}\right|,
	\end{equation}
	for all $(\tau,\xi)=(\tau_1,\ldots,\tau_n,\xi)\in\mathbb{Z}^n\times \mathbb{N}$ and all $j=1,\ldots,n$.

	If the system is not globally hypoelliptic,
	it follows from \eqref{eqqq1} that there exists a sequence $(\tau_{\ell},\xi_{\ell})=(\tau_{1\ell},\ldots,\tau_{n\ell},\xi_{\ell})\in\mathbb{Z}^n\times\mathbb{N}$ such that
	\begin{equation*}
	\max_{j=1,\ldots,n}\left|\dfrac{\tau_{j\ell}}{\xi_{\ell}^{\kappa}}+\alpha_j\right|<(|\tau_{\ell}|+\xi_{\ell})^{-\rho\ell},\;\; |\tau_\ell|+\xi_\ell\geq \ell.
	\end{equation*}
	
	For each $j=1,\ldots,n$ we observe that  $\sum_{m=1}^{\eta}\left|({\tau_{j\ell}}/{\xi_{\ell}^{\kappa}})^{\eta-m}\alpha_j^{m-1}\right|$ converges to $\eta|\alpha_j|^{\eta-1}$ as $\ell$ goes to infinity. Consequently,  this sum is bounded by a constant  $C_j>0$. If $C=\max\{C_j,\;j=1,\ldots,n\}$, then the previous inequality together with \eqref{nice equality} imply    
	\begin{eqnarray*}
		\max_{j=1,\ldots,n}\left|\dfrac{(\tau_{j\ell})^{\eta}}{\xi_{\ell}^{\rho}}+\alpha_j^\eta\right|
		\leq C (|\tau_{\ell}|+\xi_{\ell})^{-\rho\ell}\leq C (\xi_{\ell}^{\rho})^{-\ell} ,
	\end{eqnarray*}
	for all $\ell\in \mathbb{N}$. Therefore,  $\alpha$ satisfies  $(SDA)_\eta$.\\
	
	Assume now that $\alpha$ satisfies  $(SDA)_\eta$. Thus, there exist $C_0>0$ and a sequence $(r_\ell,s_\ell)=(r_{1\ell},\ldots,r_{n\ell},s_\ell)\in\mathbb{Z}^n\times\mathbb{N}$ satisfying the following inequality
	$$
	\max_{j=1,\ldots, n}	\left|\dfrac{(r_{j\ell})^ {\eta}}{s_\ell}+\alpha_j^{\eta}\right|\leq C_0 (|r_\ell|+s_\ell)^{-\ell},\quad \; \ell\in\mathbb{N}.
	$$		
	
	We claim that the system is not  globally hypoelliptic. Indeed, suppose otherwise by contradiction. 
	Since gcd$(\rho,\eta)=1$, there exist $\widetilde{p}, \widetilde{q}\in\mathbb{N}$ such that $\widetilde{p}\rho-\widetilde{q}\eta=1$. Then, for each $j=1,\ldots,n$, the equality bellow follows from  \eqref{nice equality}:
	\begin{align*}
	\left|\dfrac{(r_{j\ell})^{\eta}}{s_\ell}+\alpha_j^{\eta}\right | = \left|\dfrac{(r_{j\ell} s_{\ell}^{\widetilde{q}})^{\eta}}{s_\ell^{\rho\widetilde{p}}}+\alpha_j^{\eta}\right|  
	&= \left|\dfrac{r_{j\ell} s_{\ell}^{\widetilde{q}}}{(s_{\ell}^{\widetilde{p}})^{\kappa}}+\alpha_j\right| \cdot  \left|\sum_{m=1}^{\eta} \left(\dfrac{r_{j\ell} s_{\ell}^{\widetilde{q}}}{(s_{\ell}^{\widetilde{p}})^{\kappa}}\right)^{\eta-m}\alpha_j^{m-1}\right| \\
	& = \left|\dfrac{r_{j\ell} s_{\ell}^{\widetilde{q}}}{(s_{\ell}^{\widetilde{p}})^{\kappa}}+\alpha_j\right| \cdot  \left|\sum_{m=1}^{\eta} \left(\dfrac{(r_{j\ell})^{\eta} }{s_{\ell}}\right)^{({\eta-m})/\eta}\alpha_j^{m-1}\right|. 
	\end{align*}
	
	Since $(r_{j\ell})^{\eta}/{s_\ell}$ converges to $\alpha_j^\eta$, for $\alpha_{j}\neq 0$ there exists a small $\varepsilon_j>0$ such that 
	$$
	\left|\sum_{m=1}^{\eta} \left(\dfrac{(r_{j\ell})^{\eta} }{s_{\ell}}\right)^{({\eta-m})/\eta}\alpha_j^{m-1}\right|>\eta|\alpha_j|^{\eta-1}- \varepsilon_j >0,
	$$ 
	for large enough $\ell$. On the other hand, if $\alpha_{j}=0$, then $r_{j\ell}=0$ for large enough $\ell$.
	
	We conclude that 
	\begin{align*}
	\widetilde{C} \max_{j=1,\ldots, n}\left|\dfrac{(r_{j\ell}) ^{\eta}}{s_\ell}+\alpha_j^{\eta}\right | 
	\geq \max_{j=1,\ldots, n}\left|\dfrac{r_{j\ell} s_{\ell}^{\widetilde{q}}}{(s_{\ell}^{\widetilde{p}})^{\kappa}}+\alpha_j\right|,
	\end{align*}
	where  $\widetilde{C}=\max_{\alpha_{j}\neq 0}\left(\eta|\alpha_j|^{\eta-1}-\varepsilon_j\right)^{-1}>0$ and $\ell$ is large enough.  
	
	Finally, the previous inequality together with  \eqref{eqqq1} imply that  
	\begin{align*}
	{C_0\widetilde{C}}(|r_\ell|+s_\ell)^{-\ell}&\geq C (|r_{1\ell} s_{\ell}^{\widetilde{q}}|+\cdots+ |r_{n\ell} s_{\ell}^{\widetilde{q}}|+ s_{\ell}^{\widetilde{p}})^{-M}\\
	&= C (s_{\ell}^{\widetilde{q}}|r_\ell|+s_{\ell}^{\widetilde{p}})^{-M}.
	\end{align*}
	
	Hence
	$$
	\dfrac{(|r_\ell|+s_\ell)^{\ell}}{(s_{\ell}^{\widetilde{q}}|r_\ell|+s_{\ell}^{\widetilde{p}})^M}
	\leq
	\frac{C_0\widetilde{C}}{C},
	$$
	for large enough $\ell$, which is a contradiction since the sequence in the left hand side of the above inequality goes to infinity.\\

	The case in which ${\beta}\neq 0$ and   $\widetilde{\beta}=0$ is completely analogous to the previous one. If ${\beta}= 0$ and   $\widetilde{\beta}=0$ we have that the system is globally hypoelliptic if and only if $\alpha$ and $\widetilde{\alpha}$ belong to $\mathbb{R}^n\setminus\mathbb{Q}^n$ and both do not satisfy $(SDA)_\eta$. The proof in this case will be omitted, since it can be obtained by a slight modification in the previous arguments.
	
\end{proof}

The following example is an immediate consequence of Theorem \ref{homogeneous hypoellipticity}.

\begin{example} Let $\eta\geq 2$ be an integer. 
	Consider the vector $\alpha\in\mathbb{R}^2\setminus\mathbb{Q}^2$ obtained in Example \ref{SDA example}. Thus, the system 
	$$
	L_j=D_{t_j}+\alpha_{j}(D_x^{2})^{1/2},\quad j=1,2,
	$$
	is globally hypoelliptic on $\mathbb{T}_t^2\times\mathbb{T}_x^1$, while the system
	$$
	L_j=D_{t_j}+\alpha_{j}(D_x^{2})^{1/2\eta},\quad j=1,2,
	$$
	is not.
\end{example}

%===========================================================================
%===========================================================================
\section{Systems with variable coefficients}\label{Systems with variable coefficients}
%===========================================================================
%===========================================================================
We now begin to study the global hypoellipticity of system  \eqref{system}, which has variable coefficients. Often, we will denote this system by $\mathbb{L}$.

The key point in our analysis is the study  of the behavior of the functions
\begin{equation}\label{functions-M-j}
t_{j} \in \mathbb{T}^1 \mapsto \mathcal{M}_j(t_j, \xi) = c_j(t_j)p_j(\xi),\quad \xi \in \mathbb{T}^N,
\end{equation}
and its averages
$\mathcal{M}_{j0}(\xi) = c_{j0}p_j(\xi)$, where 
$$
c_{j0} = (2\pi)^{-1} \int_{0}^{2\pi}c_j(s) ds.
$$

Based on these notations we consider the following sets 
\begin{equation}\label{Z sets}
\mathcal{Z}_j = \{\xi \in \mathbb{Z}^N; \ \mathcal{M}_{j0}(\xi) \in \mathbb{Z} \}
\textrm{ and } \mathcal{Z} = \bigcap_{j=1}^{n} \mathcal{Z}_j.
\end{equation}

We will present some connections between the global hypoellipticity of system  $\mathbb{L}$ and the global hypoellipticity of its \textit{normal form} $\mathbb{L}_0$ defined in \eqref{normal-form}.

In order to better present these connections, we separate our results in two subsections: Necessary conditions \ref{sec necessary conditions} and Sufficient conditions \ref{sec-suff-cond}.  Before that,  we recall the following results:

\begin{theorem}[See \cite{BCM93}]\label{The-diff-system}
	Consider the system 
	\begin{equation}\label{system-dfifferential}
	L_j = D_{t_j} + (a_j(t_j)+ ib_j(t_j)) D_x, \quad j=1,\ldots,n,
	\end{equation}	
defined on $\mathbb{T}_t^n
\times \mathbb{T}_x^1$ and set 
	$$
	J=\{j\in\{1,\ldots,n\};\; b_j\equiv0\}=\{j_1< \ldots< j_k\}.
	$$  	
	Then, the system \eqref{system-dfifferential} is globally hypoelliptic if and only if at least one of the follow\-ing conditions holds:	
	\begin{enumerate}
		\item[a)] There exists $b_j\not\equiv 0$ that does not change sign; 
		\item[b)] If $J\neq \emptyset$, then $(a_{j_10}, \ldots, a_{j_k0})$ is neither rational nor a Liouville vector.
	\end{enumerate}
	
\end{theorem}

\begin{proposition}\label{prop-smooth}
	Let $u \in \mathcal{D}'(\mathbb{T}^{n+N})$ be a distribution and  its $x$-Fourier coeffi\-cients
	\begin{equation*}
	\widehat{u}(\cdot, \xi) = (2 \pi)^{-n} <u(\cdot, x),  e^{-i x\xi}>,  \quad
	\xi \in \mathbb{Z}^N.
	\end{equation*}
	
	Given a sequence $\{c_{\xi}(\cdot)\}_{\xi \in \mathbb{Z}^N}$ of smooth functions on $\mathbb{T}^n$, the formal series $u = \sum_{\xi \in \mathbb{Z}^{N}}{c_\xi(t) e^{i  x \cdot \xi}}$ 	converges in $\mathcal{D}'(\mathbb{T}^{n+N})$ if and only if for any $\alpha \in \mathbb{Z}_+^n$ there exist positive constants $C$, $M$ and $R$ such that
	\begin{equation}\label{part-smooth-coef}
	|\partial^{\alpha}_t c_{\xi}(t)| \leq C \|\xi\|^{-M}, \quad \|\xi\|\geq R.
	\end{equation}
	Moreover, $u \in C^{\infty}(\mathbb{T}^{n+N})$ if and only if \eqref{part-smooth-coef} holds true for every $M>0$. In both cases, 
	$c_{\xi}(\cdot) = \widehat{u}(\cdot, \xi)$.
\end{proposition}

%===========================================================================
%===========================================================================
\subsection{Necessary conditions}\label{sec necessary conditions}
%===========================================================================
%===========================================================================

This section is intended to provide necessary conditions   to the global hypoellipticity of system $\mathbb{L}$. The first result, Proposition \ref{Z infinity}, exhibits a number-theoretical condition, while Theorem \ref{necessary condition} shows that the global hypoellipticity of system $\mathbb{L}_0$ is a necessary condition.

\begin{proposition}\label{Z infinity} If $\mathbb{L}$ is globally hypoelliptic, then  the set  $\mathcal{Z}$, defined in \eqref{Z sets}, is finite.
\end{proposition}

\begin{proof}
	Admit that $\mathcal{Z}$ is an infinite set and consider a sequence $\{\xi_{\ell}\}\in \mathcal{Z}$ such that $\|\xi_{\ell}\|$ is increasing. Thus, we can define the following smooth functions on $\mathbb{T}^n$: 
	$$
	\widehat{u}(t,\xi_{\ell})=K_\ell \prod_{j=1}^{n}\exp\left(-i\int_{0}^{t_j}\mathcal{M}_j(r,\xi_{\ell})dr\right),\quad \ell\in\mathbb{N},
	$$  
	where the constants $K_\ell$ are chosen as follows: for each $j\in\{1,\ldots,n\}$ consider the sequence $\{\tau_{\ell}^{(j)}\}$ defined  by
	$$
	\int_{0}^{\tau_{\ell}^{(j)}} \Im \mathcal{M}_j(r,\xi_{\ell})dr\doteq \max_{\tau\in[0,2\pi]} 	\int_{0}^{\tau} \Im \mathcal{M}_j(r,\xi_{\ell})dr,\quad \ell\in\mathbb{N},
	$$ 
	and set  
	$$
	K_\ell= \prod_{j=1}^{n}\exp\left(-\int_{0}^{\tau_{\ell}^{(j)}}\Im\mathcal{M}_j(r,\xi_{\ell})dr\right),\quad \ell\in\mathbb{N}.
	$$
	
	Note that $|\widehat{u}(t,\xi_{\ell})|\leq 1$ for all $t\in\mathbb{T}^n$ and for all $\ell\in\mathbb{N}$. Also, by writing $\tau_{\ell}=(\tau_{\ell}^{(1)},\ldots,\tau_{\ell}^{(n)})$, we have $|\widehat{u}(\tau_{\ell},\xi_{\ell})|=1$ for all $\ell\in\mathbb{N}$, which implies:
	$$
	u(t,x)\doteq\sum_{\ell=1}^{\infty} \widehat{u}(t,\xi_{\ell})e^{i\xi_{\ell} \cdot x}\in \mathcal{D}'(\mathbb{T}^{n+N})\setminus C^{\infty}(\mathbb{T}^{n+N}).
	$$
	
	On the other hand,  coefficients $\widehat{u}(t,\xi_{\ell})$ satisfy  equations
	$$
	\partial_{t_j}\widehat{u}(t,\xi_\ell)+i\mathcal{M}_j(t_j,
	\xi_\ell)\widehat{u}(t,\xi_\ell)=0,\quad j=1,\ldots,n,
	$$
	which implies  $L_ju=0$ for each $j=1,\ldots,n$. Therefore, $\mathbb{L}$ is not globally hypoelliptic.
	
\end{proof}

	\begin{remark}
		Proposition \ref{Z infinity} recovers the well known result on the non global hypoellipticity for a differential system 
		$$
		L_j = D_{t_j} + a_j(t_j) D_x, \quad j=1,\ldots,n,
		$$
		defined on $\mathbb{T}_t^n\times \mathbb{T}_x^1$, where each $a_j(t_j)$ is a real-valued function with rational  average $a_{j0}$. In fact, we have that $q \mathbb{Z}\subset \mathcal{Z}$ for any positive integer $q$, such that $q a_{j0}\in\mathbb{Z}$ for all $j=1,\ldots,n$.
	\end{remark}

\begin{theorem}\label{necessary condition} If  $\mathbb{L}$ is globally hypoelliptic, then $\mathbb{L}_0$ is also globally hypoelliptic.
\end{theorem}

\begin{proof} Suppose that $\mathbb{L}_0$ is not globally hypoelliptic. 
	By a slight modification on the proof of Lemma 3.1, in \cite{BDG}, we may obtain a sequence $\{\xi_\ell\}$ such that $\|\xi_{\ell}\|$ is increasing and 
	\begin{equation}\label{inequality1}
	\max_{ j=1,\ldots,n} \left|1-\exp(- 2\pi i \mathcal{M}_{j0}(\xi_{\ell}))\right|< \|\xi_{\ell}\|^{-{\ell}}, 
	\quad \ell\in\mathbb{N}.
	\end{equation}

	Now, by Proposition \ref{Z infinity} we may assume that $\mathcal{Z}=\cap_{j=1}^{n}\mathcal{Z}_j$ is finite and  $\xi_{\ell}\notin\mathcal{Z}$ for all $\ell\in\mathbb{N}$. In other words,  for each $\ell$, at least one  $j\in\{1,\ldots,n\}$ satisfies $\xi_{\ell}\notin \mathcal{Z}_{j}$. In particular, there exist $k\in\{1,\ldots,n\}$ such that $\mathcal{M}_{k0}(\xi_{\ell})\notin \mathbb{Z}$ 	for infinitely many index $\ell$.
	
	Now, consider the sequence $t_{\ell}=(t_{1\ell},\ldots,t_{n\ell})\in[0,2\pi]^n$ where each $t_{j\ell}$ is defined by 	
	$$
	\int_{0}^{t_{j\ell}} \Im \mathcal{M}_j(r,\xi_{\ell})dr=	\max_{\tau\in[0,2\pi]}\int_{0}^{\tau} \Im \mathcal{M}_j(r,\xi_{\ell})dr.
	$$
	Therefore, for each $j=1,\ldots,n$ and each $\ell$ we have 
	\begin{equation}\label{inequality2} 
	\int_{t_{j\ell}}^{t_j} \Im \mathcal{M}_j(r,\xi_{\ell})dr= \int_{0}^{t_j} \Im \mathcal{M}_j(r,\xi_{\ell})dr- \int_{0}^{t_{j\ell}} \Im \mathcal{M}_j(r,\xi_{\ell})dr \leq 0,
	\end{equation}
	for all  $t_j\in[0,2\pi]$. 
	
	By taking a subsequence, if necessary, we may assume that $t_\ell$ converges to $t_0=(t_1^{0},\ldots,t_n^{0})\in[0,2\pi]^n$.
	
	For each $j=1,\ldots,n$ we consider a compact interval $I_j\doteq [A_j,B_j]\subset(0,2\pi)$ such that $t_j^{0}\notin I_j$. Also, we take a non-negative function $\phi_j\in C_c^{\infty}(I_j;\mathbb{R})$ such that $\int_{0}^{2\pi}\phi_j(t_j)dt_j=1$.\\
	
	\centerline {\it Construction of a singular solution }
	
	Our goal from now on will be to construct a sequence of smooth coefficients $\widehat{u}(t,\xi_\ell)$ such that 	
	$$
	u(t,x)=\sum_{\ell=1}^{\infty}\widehat{u}(t,\xi_\ell) e^{i\xi_\ell\cdot x}\in \mathcal{D}'(\mathbb{T}^{n+N})\setminus C^{\infty}(\mathbb{T}^{n+N}),
	$$	
	and $f_j\doteq L_j u \in  C^{\infty}(\mathbb{T}^{n+N})$, for all $j=1,\ldots, n$.	
	
	Given $j\in\{1,\ldots,n\}$ and $\ell\in\mathbb{N}$, we will first construct auxiliar smooth coefficients $\widehat{u}_j(t_j,\xi_{\ell})$, $t_j\in[0,2\pi]$. For that, we will split the case in two: on one hand
	$\mathcal{M}_{j0}(\xi_{\ell})$ will be an integer number  and, on the other, it will not.\\

	{\it Case 1:} If $\xi_{\ell}\in \mathcal{Z}_j$, that is, $\mathcal{M}_{j0}(\xi_{\ell})\in\mathbb{Z}$, we define the following $2\pi$-periodic smooth function  
	$$
	\widehat{u}_j(t_j,\xi_{\ell})=K_{j\ell} \exp\left(-i\int_{0}^{t_j}\mathcal{M}_j(r,\xi_{\ell})dr\right),
	$$
	where the constant $K_{j\ell}$ is given as follows
	$$
	K_{j\ell}= \exp\left(-\int_{0}^{t_{j\ell}}\Im\mathcal{M}_j(r,\xi_{\ell})dr\right).
	$$
	
	In this case, it is easy to see that $|\widehat{u}_j(t_j,\xi_{\ell})|\leq 1$ for all $t_j\in[0,2\pi]$ and $|\widehat{u}_j(t_{j\ell},\xi_{\ell})|=1$.
	Furtheremore,  $i\widehat{L_j u_j}(t_j,\xi_\ell)\doteq(\partial_{t_j}+i\mathcal{M}_j(t_j,\xi_{\ell}))\widehat{u}_j(t_j,\xi_{\ell}) 
	=0$.\\

	{\it Case 2:} If $\xi_{\ell}\notin \mathcal{Z}_j$ we consider the function 
	$\widehat{g}_j(\cdot,\xi_\ell)$ as being the $2\pi-$periodic extension of the function 
	$$
	\left(1-e^{-2\pi i\mathcal{M}_{j0}(\xi_\ell)}\right)\exp\left(-\int_{t_{j\ell}}^{t_j}i\mathcal{M}_j(s,\xi_\ell)ds\right)\phi_j(t_j), \quad t_j\in[0,2\pi ].
	$$
	Therefore, in order to define the coefficient $\widehat{u}_j(\cdot,\xi_{\ell})$ we consider the equation $i\widehat{L_j u_j}(t_j,\xi_\ell)\doteq(\partial_{t_j}+i\mathcal{M}_j(t_j,\xi_{\ell}))\widehat{u}_j(t_j,\xi_{\ell}) 
	=\widehat{g}_j(t_j,\xi_\ell)$
	obtaining the following solution 
	$$
	\widehat{u}_j(t_j,\xi_\ell)= \dfrac{1}{1-e^{-2\pi i\mathcal{M}_{j0}(\xi_\ell)}}\int_{0}^{2\pi}\exp\left(-\int_{t_j-s}^{t_j}i\mathcal{M}_j(r,\xi_\ell)dr\right)\widehat{g}_j(t_j-s,\xi_\ell)ds.
	$$
	
	Note that if $j\in\{1,\ldots,n\}$ is such that $\xi_\ell\notin \mathcal{Z}_j$, for infinitely many indexes $\ell$, then $\widehat{g}_j(\cdot,\xi_\ell)$ and its derivatives dacay rapidly when $\ell$ goes to infinity. Indeed, since the toroidal symbol $p(\xi)$ increases slowly, our statement follows from inequalities \eqref{inequality1} and \eqref{inequality2}. Also, in this case,
	we claim that  $|\widehat{u}_j(\cdot,\xi_\ell)|$ is bounded when $\ell$ is sufficiently large. To verify this, we write  
	$$
	\delta_{j\ell}=(1-e^{-2\pi i\mathcal{M}_{j0}(\xi_\ell)})^{-1}
	\ \textrm{ and } \
	\Gamma_{j,\ell}(r) =\Im\mathcal{M}_j(r,\xi_\ell).
	$$
 Hence:
\begin{align*}
|\widehat{u}_j(t_j,\xi_\ell)| & \leq 
\delta_{j\ell} \Big[ \int_{0}^{t_j}\exp\left(\int_{t_j-s}^{t_j}\Gamma_{j,\ell}(r) dr\right)|\widehat{g}_j(t_j-s,\xi_\ell) |ds \, + \\ 
& + \int_{t_j}^{2\pi}\exp\left(\int_{t_j-s}^{t_j}\Gamma_{j,\ell}(r) dr \right)|\widehat{g}_j(t_j-s+2\pi,\xi_\ell) | ds\Big] \\ 
& =  \int_{0}^{t_j}\exp\left(\int_{t_{j\ell}}^{t_j} \Gamma_{j,\ell}(r) dr\right)\phi_j(t_j-s) ds \, + \\ 
& +\int_{t_j}^{2\pi}\exp\left(\int_{t_j-s}^{t_j}\Gamma_{j,\ell}(r) dr +\int^{t_j-s+2\pi}_{t_{j\ell}}\Gamma_{j,\ell}(r) dr\right) 
\phi_j(t_j-s+2\pi) ds  \\
& =  \int_{0}^{t_j}
\exp\left(\int^{t_j}_{t_{j\ell}}\Gamma_{j,\ell}(r) dr\right)\phi_j(t_j-s) ds \, + \\ 
& + \int_{t_j}^{2\pi}\exp\left(\int^{t_j}_{t_{j\ell}}\Gamma_{j,\ell}(r) dr+2\pi \Im\mathcal{M}_{j0}(\xi_\ell)\right)\phi_j(t_j-s+2\pi) ds .
\end{align*}

It follows from \eqref{inequality1} that $\Im\mathcal{M}_{j0}(\xi_\ell)\rightarrow 0$, since $\Im\mathcal{M}_{j0}(\xi_\ell)\notin \mathbb{Z}$. Thus, there exist $\epsilon>0$ such that $\exp({2\pi \Im\mathcal{M}_{j0}(\xi_\ell) })\leq 1+\epsilon$, when $\ell$ is large enough. Then,  inequality  \eqref{inequality2} together with the  previous one imply, for sufficiently large  $\ell$, the following:
$$
|\widehat{u}_j(t_j,\xi_\ell)|\leq 2+\epsilon,\quad \forall t_j\in[0,2\pi] .
$$
	
	Now, let us  verify that  $|\widehat{u}_j(t_{j\ell},\xi_{\ell})|\rightarrow 1$. Firstly, 
	if $t_{j}^{0}> B_j$, then, for sufficiently large $\ell$, we have $t_{j\ell}>B_j$   and 
	$$
	|\widehat{u}(t_{j\ell},\xi_\ell)|
	=\int_{0}^{2\pi}\phi_j(t_{j\ell}-s)ds=\int_{0}^{t_{j\ell}}\phi_j(\tau)d\tau= 1.
	$$ 
	
	On the other hand, if $t_{j}^{0}<A_j$, then $t_{j\ell}<A_j$, for  $\ell$ sufficiently large, and 
	\begin{eqnarray*}
		|\widehat{u}_j(t_{j\ell},\xi_\ell)|&=& \Big|\int_{0}^{2\pi}\exp\left(-\int_{t_{j\ell}-s}^{t_{j\ell}}i\mathcal{M}_j(r,\xi_\ell)dr \right)
		\\
		&\times&\exp\left(- \int^{t_{j\ell}-s+2\pi}_{t_{j\ell}}i\mathcal{M}_j(r,\xi_\ell)dr\right) \phi_j(t_{j\ell}-s+2\pi)ds\Big|\\
		&=&\exp(2\pi \Im \mathcal{M}_{j0}(\xi_{\ell}))\int_{0}^{2\pi} \phi_j(t_{j\ell}-s+2\pi)ds \\
		&=& \exp(2\pi \Im\mathcal{M}_{j0}(\xi_{\ell})) \int_{0}^{2\pi} \phi_j(t_{j\ell}-s)ds\\
		&=& \exp(2\pi \Im\mathcal{M}_{j0}(\xi_{j\ell})) 
		\int_{t_{j\ell} - 2 \pi}^{t_{j\ell}} \phi_j(s)ds		\\
		&=& \exp(2\pi \Im\mathcal{M}_{j0}(\xi_{\ell})).
	\end{eqnarray*}
Hence $|\widehat{u}_j(t_{j\ell},\xi_{\ell})| \rightarrow 1$, when $\ell$ goes to infinity.\\

	\centerline{\it Final part}
	
	We define 
	$$
	\widehat{u}(t,\xi_{\ell})=\widehat{u}(t_1,\ldots,t_n,\xi_{\ell})\doteq\prod_{j=1}^{n} \widehat{u}_j(t_j,\xi_{\ell}). 
	$$
	
	Thus 
	$
	|\widehat{u}(t,\xi_{\ell})|\leq (2+\epsilon)^n
	$
	for all $t\in \mathbb{T}^n$ and for $\ell\in\mathbb{N}$ sufficiently large. Hence, 
	$$
	u(t,x)=\sum_{\ell=1}^{\infty} \widehat{u}(t,\xi_{\ell}) e^{i x \xi_{\ell}}\in \mathcal{D}'(\mathbb{T}^{n+N}).
	$$
	
	It follows from the previous construction  that 
	$|\widehat{u}(t_\ell,\xi_{\ell})|$ converges to $1$ when $\ell$ goes to infinity, which implies that $u(t,x) $ is not a smooth function  on $\mathbb{T}^{n+N}$. \\
	
	Finally,  note that for each $k\in\{1,\ldots,n\}$ we have 
	\begin{align*}
	\widehat{f}_k(t,\xi_\ell) \doteq \widehat{L_k u}(t,\xi_{\ell})
	& = \left(\prod_{j=1, j\neq k}^n \widehat{u}_j(t_j,\xi_\ell) \right) \widehat{L_k u_k}(t_k,\xi_\ell) \\
	& = \left\{
	\begin{array}{l}
	\left(\prod_{j=1, j\neq k}^n \widehat{u}_j(t_j,\xi_\ell) \right) \widehat{g}_k(t_k,\xi_{\ell}),  \ \textrm{ if } \ 
	\xi_{\ell}\notin\mathcal{Z}_k, \\
	0, \ \textrm{ if } \  \xi_{\ell}\in\mathcal{Z}_k.
	\end{array}
	\right.
	\end{align*}

	Since $\widehat{g}_k(t_k,\xi_{\ell}) $ decays rapidly and for each $j$ we have $|\widehat{u}_j(t_j,\xi_\ell)|\leq 2+\epsilon$, for all $t_j$ and all $\ell$ sufficiently large. Then, $\widehat{f}_k(t,\xi_\ell)$ decays rapidly. It follows that $f_k\in C^{\infty}(\mathbb{T}^{n+N})$.

\end{proof}

%===========================================================================
%===========================================================================
\subsection{Sufficient conditions} \label{sec-suff-cond}
%===========================================================================
%===========================================================================

The starting point here is the fact that the reciprocal of Theorem \ref{necessary condition} is not true. For instance, consider  system $\mathbb{L}$ given by
$$
L_j = D_{t_j} + i \, b_j(t_j) D_x, \quad j=1,\ldots,n,
$$
where each $b_j$ is a real-valued function that changes sign and, additionally, suppose that there exists an average $b_{k0}$ that is not zero. Thus, operator $L_ {k0}$ is globally hypoelliptic on $\mathbb{T}_{t_k}^{1}\times \mathbb{T}_x^{1}$,  
which implies that the same happens with  system  
$\mathbb{L}_0$ (see Corollary \ref{coro-const-system}). However, in view of Theorem \ref{system-dfifferential},  system $\mathbb{L}$ is not globally hypoelliptic. 

Regarding  system $\mathbb{L}$, given in \eqref{system},  since  the existence of a globally hypoelliptic operator $L_ {j0}$  in $\mathbb{T}_{t_j}^{1}\times \mathbb{T}_x^{N}$ implies the global hypoellipticity of $\mathbb{L}_0$, it is natural to look for additional conditions on just the operator $L_j$ in order to obtain the global hypoellipticity of $\mathbb{L}$.

In order to state the main result in this section, it will be necessary to establish some preliminary definitions. We will now move on to describing them. Firstly, motivated by \cite{Avila}, we consider the following  H\"{o}rmander condition for a real-valued function $\psi$ defined on $\mathbb{T}^1\times\mathbb{Z}^N$ (see Chapter XXIII in \cite{Hormander3}): there exists a positive constant $\Theta$ such that
\begin{equation}\label{Hormander condition}
\sup_{t\in\mathbb{T}^1}\psi(t,\xi)\geqslant - \Theta,\quad  \forall \xi \in \mathbb{Z}^N.
\end{equation}
Furthermore, taking inspiration from \cite{AGKM}, we say that a  function $\phi:\mathbb{Z}^N\rightarrow \mathbb{C}$ has at most logarithmic growth,
if there are positive constants $\kappa$ and $n_0$ such that
\begin{equation}\label{log-estimate-1}
|\phi(\xi)| \leq  \kappa \log(\|\xi\|), \quad \forall \|\xi\|\geq  n_0.
\end{equation}

By simplicity, when this property is verified we write $\phi(\xi) = O(\log(\|\xi\|))$. When it 
 is not, we say that $\phi(\xi)$ has \textit{super-logarithmic growth}.\\

Now, with respect to the system under study, for each $j\in\{1,\ldots,n\}$,  we write the symbol $p_j(\xi)=\alpha_j(\xi)+i\beta_j(\xi)$ and the correspondent coefficient as $c_j(t) = a_j(t_j) + i b_j(t_j)$.

\begin{definition}\label{growth}
    Consider the functions 
	$$
	\Im \mathcal{M}_j(t_j,\xi)=a_j(t_j)\beta_j(\xi)+b_j(t_j)\alpha_j(\xi).
	$$
	We denote by $\mathcal{H}$ the set of $j\in \{1, \ldots, n\}$  such that 	$\Im \mathcal{M}_j$ satisfies condition \eqref{Hormander condition} and by $\mathcal{L}$  the set of $j\in \{1, \ldots, n\}$ such that one the following conditions hold:
	\begin{enumerate}
		\item [(i)] $p_j(\xi) = O(\log(\|\xi\|))$;
		
		\item [(ii)] $\alpha_j(\xi) = O(\log(\|\xi\|))$,  $\beta_j(\xi)$ has super-logarithmic growth and  $a_j(\cdot)$ does not change sign;
		
		\item [(iii)] $\alpha_j(\xi)$ has super-logarithmic growth, $\beta_j(\xi) = O(\log(\|\xi\|))$ and $b_j(\cdot)$ does not change sign.
	\end{enumerate}
\end{definition}

We are now ready to state the main result in this section.
\begin{theorem}\label{main-suff}
	If there exists $j\in\{1,\ldots,n\}$ such that $L_{j0}$ is globally hypoelliptic on $\mathbb{T}_{t_j}^{1}\times\mathbb{T}_x^{N}$ and $j\in \mathcal{H}\; {\cup} \;\mathcal{L}$, 	then the system $\mathbb{L}$ is globally hypoelliptic.
\end{theorem}

\begin{remark}
We recall that if  operator $L_j$ is globally hypoelliptic, then  the same occurs to $L_{j0}$. Therefore, Theorem \ref{main-suff} can be rewritten just by repla\-cing the hypothesis ``$L_{j0}$ is globally hypoelliptic" for  ``$L_{j}$ is globally hypoelliptic".
\end{remark}

The strategy for proving Theorem \ref{main-suff} will be the following: firstly, in Proposition \ref{general-suff-cond} we exhi\-bit   an abstract condition that, in addition with the global hypoellipticity of some $L_{j0}$, ensures that system $\mathbb{L}$ is globally hypoelliptic. Hence, the proof of Theorem \ref{main-suff}  will be given as a consequence of Propositions \ref{prop-hormander-system}, \ref{prop_log} and \ref{pro-super-log}. Each one of these propositions will be  presented in a separate subsection and, in addition, several examples will be provided in order to clarify our results. 

\begin{proposition}\label{general-suff-cond}
	Admit that operator $L_{j0}$ is globally hypoelliptic. If there exist positive constants $C, M$ and $R$ such that either
	\begin{equation*}
	\sup_{(\zeta,\tau) \in \mathbb{T}^2} \left[\exp \left(\int_{\zeta-\tau}^{\zeta}\Im \mathcal{M}_j(s,\xi)ds \right)\right] \leq C\|\xi\|^{M} , \ \|\xi\| \geq R,
	\end{equation*}
	or
	\begin{equation*}
	\sup_{(\zeta,\tau) \in \mathbb{T}^2} \left[\exp \left(-\int_{\zeta}^{\zeta+\tau}\Im \mathcal{M}_j(s,\xi)ds \right)\right] \leq C\|\xi\|^{M},   \ \|\xi\| \geq R,
	\end{equation*}
	then the system $\mathbb{L}$  is global hypoelliptic.
\end{proposition}
\begin{proof} Let $u \in \mathcal{D}'(\mathbb{T}^{n+N})$ be a solution  of the equations
	\begin{equation*}
	iL_ku=f_k\in C^\infty(\mathbb{T}^{n+N}), \quad k =1,\ldots, n.
	\end{equation*} 
	
Let $j\in  \{1,\ldots, n\}$ be so that the hypothesis are fulfilled. We will show that, in fact, $u$ is a smooth function.

	By  using the $x$-Fourier series
	$$
	u(t, x) = \sum_{\xi \in \mathbb{Z}} \widehat{u}(t,\xi) e^{i \xi x}
	\text{ \ and \ }
	f_j(t, x) = \sum_{\xi \in \mathbb{Z}} \widehat{f}_{j}(t,\xi) e^{i \xi x},
	$$
	it follows that equation $iL_j u=f_j$ is equivalent to the  ordinary differential equations
	\begin{equation}\label{equations-2}
	\partial_{t_j}\widehat{u}(t,\xi)+i\mathcal{M}_j(t_j,\xi)\widehat{u}(t,\xi)=\widehat{f}_j(t,\xi),\quad t\in\mathbb{T}^n, 
	\end{equation}
	for every $\xi\in\mathbb{Z}^N$.

	For each $\xi \notin \mathcal{Z}_j$,  equation \eqref{equations-2} has exactly one solution which can be written in the following two equivalent ways:	
	\begin{equation}\label{solutions-1}
	\widehat{u}(t,\xi)=\dfrac{1}{{1-e^{-2\pi i\mathcal{M}_{j0}(\xi)}}}\int_{0}^{2\pi}\mathcal{H}_j(t,\tau,\xi)
	\widehat{f}_j(t_1,\ldots,t_j-\tau,\ldots,t_n,\xi)d\tau
	\end{equation}
	and
	\begin{equation}\label{solutions-2}
	\widehat{u}(t,\xi)=\dfrac{1}{{e^{2\pi i\mathcal{M}_{j0}(\xi)} - 1}}
	\int_{0}^{2\pi}\mathcal{G}_j(t,\tau,\xi)
	\widehat{f}_j(t_1,\ldots,t_j+\tau,\ldots,t_n,\xi)d\tau,
	\end{equation}
	where 
	\begin{equation*}
	\mathcal{H}_j(t,\tau,\xi)=
	\exp \left(-i\int_{t_j-\tau}^{t_j}\mathcal{M}_j(s,\xi)ds\right), \quad \xi \in \mathbb{Z}^N
	\end{equation*}
	and
	\begin{equation*}
	\mathcal{G}_j(t,\tau,\xi)= 
	\exp\left( i\int_{t_j}^{t_j+\tau}\mathcal{M}_j(s,\xi)ds\right), \quad \xi \in \mathbb{Z}^N. 
	\end{equation*}
	
	Since $L_{j0}$ is globally hypoelliptic,  the corresponding set $\mathcal{Z}_j$  is finite and there are positive constants $C_1$, $M_1$ and $R_1$ such that 
	\begin{equation}\label{l0j-GH}
	| 1-e^{- 2\pi i\mathcal{M}_{j0}(\xi)} |\geq C_1\|\xi\|^{-M_1},
	\quad  \|\xi\|\geq R_1.
	\end{equation} 
	
	Also, since  function $f_j$  is smooth, given 
	$\gamma \in \mathbb{Z}_+^n$ and $\widetilde{N}>0$, we obtain  positive constants $C_2$ and $R_2$ satisfying
	\begin{equation}\label{f_j-smooth}
	\sup_{t \in \mathbb{T}^n}  |\partial_{t}^{\gamma} \widehat{f}_j(t, \xi)| \leq C_2 \|\xi\|^{-\widetilde{N}}, \quad \|\xi\|\geq R_2.
	\end{equation}

Hence, given  $\alpha =(\alpha_1, \ldots, \alpha_n) \in \mathbb{Z}^n_+$,  we obtain, from Leibniz's formula, equations \eqref{solutions-1}, \eqref{l0j-GH} and \eqref{f_j-smooth} that
	\begin{equation}\label{ineq-derivative-uj}
	|\partial_t^{\alpha}  \widehat{u}(t, \xi)| \leq
	C'\|\xi\|^{M_1 + \alpha_j m_j - \widetilde{N}}
	\int_{0}^{2\pi} \exp \left(\int_{t_j-\tau}^{t_j}\Im \mathcal{M}_j(s,\xi)ds \right)d\tau,
	\end{equation}
	for $\|\xi\|$ large enough with $C' = C_1C_2C_3$, where the constant $C_3>0$, obtained using \eqref{bound-symb}, depends only on $\alpha$ and the function $c_j$.

	On the other hand, by similar arguments, it follows from  \eqref{solutions-2} that
	\begin{equation}\label{ineq-derivative-uj-2}
	|\partial_t^{\alpha}  \widehat{u}(t, \xi)| \leq
	C'\|\xi\|^{M_1 + \alpha_jm_j - \widetilde{N}}
	\int_{0}^{2\pi} \exp \left(-\int_{t_j}^{t_j+\tau}\Im \mathcal{M}_j(s,\xi)ds \right)d\tau.
	\end{equation}

	Finally, we obtain from the hypothesis, inequalities \eqref{ineq-derivative-uj} and \eqref{ineq-derivative-uj-2} that
	$$
	|\partial_t^{\alpha}  \widehat{u}(t, \xi)| \leq
	C'C\|\xi\|^{M + M_1 + \alpha_j m_j - \widetilde{N}}.
	$$
	
	Therefore, it follows from Proposition \ref{prop-smooth}  that $u$ is a smooth function on $\mathbb{T}^{n+N}$, which concludes the proof.
	
\end{proof}

%===========================================================================
%===========================================================================
\subsubsection{H\"{o}rmander condition \label{sec-H-cond}} 
%===========================================================================
%===========================================================================

In the next proposition we prove the global hypoellipticity for  system $\mathbb{L}$ by assuming that some $L_{j0}$ is globally hypoelliptic and $\Im \mathcal{M}_j$ satisfies  H\"{o}rmander's condition \eqref{Hormander condition}.

\begin{proposition}\label{prop-hormander-system}
	If there exists $j \in \mathcal{H}$ such that $L_{j0}$  is globally hypoelliptic on $\mathbb{T}_{t_j}^1\times\mathbb{T}_x^N$, then  system $\mathbb{L}$ is globally hypoelliptic.
\end{proposition}

\begin{proof}
	Let $u$ be a distribution  on $\mathbb{T}^{n+N}$ such that  
	$$
	iL_ju=f_j\in C^\infty(\mathbb{T}^{n+N}),
	$$
	for each $j \in \{1,\ldots, n\}$ and consider  solutions 
	\eqref{solutions-2} and  estimate \eqref{ineq-derivative-uj-2}. 
	
	By hypothesis \eqref{Hormander condition} we get
	$$
	\int_{0}^{2\pi} \exp \left(-\int_{t_j}^{t_j+\tau}\Im \mathcal{M}_j(s,\xi)ds \right)d\tau  \leq \exp(2 \pi \Theta_j).
	$$
	
	Therefore, by Proposition \ref{general-suff-cond}, system $\mathbb{L}$ is globally hypoelliptic.
	
\end{proof}

\begin{remark}
It is important to point out that  condition  \eqref{Hormander condition} can be replaced by 
\begin{equation*}\label{Horm_cond-system-2}
\Im \mathcal{M}_j(\cdot, \xi) \leq \Theta_j, \  \forall \xi \in \mathbb{Z}^n,
\end{equation*}
in view of the equivalence between equations \eqref{solutions-1} and \eqref{solutions-2}.
\end{remark}

\begin{corollary}\label{coro-nao-muda-sign}	
	Admit that 	$L_{j0}$ is globally hypoelliptic on $\mathbb{T}_{t_j}^1\times\mathbb{T}_x^N$. If $\Im \mathcal{M}_j( \cdot , \xi)$ does not change sign, for sufficiently large $\xi$, then  system $\mathbb{L}$ is globally hypoelliptic.
\end{corollary}

	We highlight that when we consider system  \eqref{system-dfifferential} and there exists $j$, such that $\Im\mathcal{M}_j(t_j, \xi) = b_j(t_j)\xi$ does not change sign for sufficiently large $\xi$, then Corollary \ref{coro-nao-muda-sign} recovers the well known {\it Nirenberg–Treves condition (P)} for the vector field $L_j$. Therefore,  system \eqref{system-dfifferential} is globally hypoelliptic.  However, this is not a necessary condition for the class of  pseudo-differential systems studied in this work, as we show in the next example.

\begin{example}\label{exe-hormander-cond}
Motivated by Example 3 in \cite{Avila}, 	let $P(D_x)$ be a pseudo-differential operator  on $\mathbb{T}^1$ with symbol $p(\xi) =\alpha(\xi) + i\beta(\xi)$ defined as follows:
	$$
	\alpha(\xi) = 
	\left\{
	\begin{array}{l}
	\xi^{-1}, \ \textrm{ if } \ \xi<0  \ \textrm{ is odd,} \\
	|\xi|, \ \textrm{ if } \ \xi \leq 0  \ \textrm{ is even,} \\
	0, \ \textrm{ if  } \ \xi > 0,
	\end{array}
	\right. 
	\ \textrm{and } \
	\beta(\xi) = 
	\left\{
	\begin{array}{l}
	1, \ \textrm{ if  } \ \xi \leq 0, \\
	\xi, \ \textrm{ if } \ \xi > 0.
	\end{array}
	\right. 
	$$

	Now, let $a(\cdot)$ and $b(\cdot)$ be two real-valued, smooth  and  positive functions on $\mathbb{T}^1$ with disjoint supports
	and define  $\mathfrak{L}= D_{s} + (a(s)+ ib(s))P(D_x)$. 
	
	Notice that	
	$$
	\Im \mathcal{M}(s,\xi) =
	\left\{
	\begin{array}{l}
	a(s) + {b(s)}/{\xi}, \ \textrm{ if } \ \xi<0  \ \textrm{ is odd,} \\
	a(s) + b(s)|\xi|, \ \textrm{ if } \ \xi \leq 0  \ \textrm{ is even,} \\
	a(s) \xi, \ \textrm{ if } \ \xi > 0,
	\end{array}
	\right.	
	$$
	then  condition \eqref{Hormander condition} is fulfilled by choosing 
	$\Theta = \max_{s \in \mathbb{T}^1} {b(s)}$. Moreover, if $\overline{s}\in supp(b)$ we have $\Im \mathcal{M}(\overline{s},\xi)<0$, for $\xi<0$ odd, and 
	$\Im \mathcal{M}(\overline{s},\xi)>0$, for $\xi<0$ even. 
	
	Now, assume $b_{0} \notin \{m a_{0}; \ m \in \mathbb{Z}\}$. It follows from
	$$
	\Im \mathcal{M}_{0}(\xi) =
	\left\{
	\begin{array}{l}
	a_{0} + {b_{0}}/{\xi}, \ \textrm{ if } \ \xi<0  \ \textrm{ is odd,} \\
	a_{0} + b_{0}|\xi|, \ \textrm{ if } \ \xi \leq 0  \ \textrm{ is even,} \\
	a_{0} \xi, \ \textrm{ if } \ \xi > 0,
	\end{array}
	\right.	
	$$
	that $\Im \mathcal{M}_{0}(\xi) \neq 0$, for all $\xi \in \mathbb{Z}$. Thus, $\mathfrak{L}_0$ is globally hypoelliptic by Proposition 3.2 in \cite{AGKM}. Hence,  any system in the form $\mathbb{L}=\{L_1, \ldots, L_{n-1},\mathfrak{L}\}$, defined on  $\mathbb{T}^{n+1}$, is globally hypoelliptic.
	
\end{example}

%===========================================================================
%===========================================================================
\subsubsection{Logarithm growth \label{sec-Log-cond}}
%===========================================================================
%===========================================================================
Now, we assume that the symbol $p_j(\xi)$ of some operator $L_j$ has at most logarithmic growth. In this case, supposing that the associated constant operator $L_{j0}$ is globally hypoelliptic, the conclusion of Theorem \ref{general-suff-cond} can be given as follows:

\begin{proposition}\label{prop_log}
	If some $L_{j0}$ is globally hypoelliptic on $\mathbb{T}_{t_j}^1\times\mathbb{T}_x^N$ and $p_j(\xi)=O(\log(\|\xi\|))$, then  system  $\mathbb{L}$ is globally hypoelliptic.
\end{proposition}

\begin{proof}
	By Proposition \ref{general-suff-cond}, it is sufficient to exhibit  positive constants $\delta$ and $R$ such that
	\begin{equation*}\label{ex_bounded_log}
\sup_{(t_j,\tau)\in\mathbb{T}^2}	\left |
	\exp \left(-i\int_{t_j-\tau}^{t_j}\mathcal{M}_j(s,\xi)ds \right) 
	\right| 
	\leq  \|\xi\|^{\delta}, \quad \forall
	\, \|\xi\|\geq R.
	\end{equation*}
	
Consider constants $\delta_{1,j}<0, \epsilon_{1,j} <0$ and $\delta_{2,j}>0 ,\epsilon_{2,j}>0$ such that
	\begin{equation}\label{boun_for_int}
	\delta_{1,j} \leq \int_{t_j - \tau}^{t_j}a_j(s) ds \leq \delta_{2,j}
	\ \textrm{ and } \
	\epsilon_{1,j} \leq \int_{t_j - \tau}^{t_j}b_j(s) ds \leq \epsilon_{2,j},
	\end{equation}
for all $(t_j,\tau)\in\mathbb{T}^2$.

	Now, let 	$\kappa_j>0$ and $n_0>0$ 	such that
	\begin{equation}\label{log_for_p}
	|\alpha_j(\xi)| \leq  \log(\|\xi\|^{\kappa_j})
	\ \textrm{ and } \
	|\beta_j(\xi)| \leq \log(\|\xi\|^{\kappa_j}),
		\end{equation}
	for all $\|\xi\|\geq  n_{0}$.

	The inequalities \eqref{boun_for_int}  and \eqref{log_for_p} imply that 
		\begin{equation}\label{alpha-b}
		\alpha_j(\xi) \int_{t_j - \tau}^{t_j}b_j(s) ds \leq   \left \{
		\begin{array}{l}
		\log(\|\xi\|^{\kappa_j \epsilon_{2,j}}), \ \textrm{ if } \  \alpha_j(\xi) >0, \\[2mm]
		\log(\|\xi\|^{-\kappa_j\epsilon_{1,j}}), \ \textrm{ if } \  \alpha_j(\xi) <0,
		\end{array}\right.
		\end{equation}
		and
		\begin{equation*}\label{beta-a}
		\beta_j(\xi) \int_{t_j - \tau}^{t_j}a_j(s) ds \leq   \left \{
		\begin{array}{l}
		\log(\|\xi\|^{\kappa_j \delta_{2,j}}), \ \textrm{ if } \  \beta_j(\xi) >0, \\[2mm]
		\log(\|\xi\|^{-\kappa_j\delta_{1,j}}), \ \textrm{ if } \  \beta_j(\xi) <0,
		\end{array}\right.
		\end{equation*}
		for all $\|\xi\|\geq  n_{0}$  and for all $(t_j,\tau)\in\mathbb{T}^2$.
	
		Hence, it follows that
	\begin{align*}
	\left |
	\exp \left(-i\int_{t_j-\tau}^{t_j}\mathcal{M}_j(s,\xi)ds \right) 
	\right| 
	& = \exp \left(\int_{t_j-\tau}^{t_j}\Im \mathcal{M}_j(s,\xi)ds \right) \\
	& = 
	\exp \left({\alpha_j(\xi) \int_{t_j - \tau}^{t_j}b_j(s) ds} + {\beta_j(\xi) \int_{t_j - \tau}^{t_j}a_j(s) ds}\right) \\
	& \leq \|\xi\|^{\delta}, %\quad \|\xi\|\geq  n_{0},
	\end{align*}
	for all $\|\xi\|\geq  n_{0}$  and for all $(t_j,\tau)\in\mathbb{T}^2$, where 
	$
	\delta = 
	\max\{\kappa_j \epsilon_{2,j}, -\kappa_j\epsilon_{1,j}\}+\max\{\kappa_j \delta_{2,j}, -\kappa_j\delta_{1,j}\}.
	$
	
	\end{proof}

\begin{example}\label{log grow example}
	Let $b$ be a real-valued function on $\mathbb{T}^1$ 
	that changes sign and  $b_0 \neq 0$. 	Consider a pseudo-differential operator $P(D_x)$ on $\mathbb{T}^1$ defined by the symbol $p(\xi) = \log(1+ |\xi|)$ and set $\mathscr{L}=D_t + ib(t)P(D_x) $. Under these conditions, any  system in the form $\mathbb{L} = \{L_1, \ldots, L_{n-1},\mathscr{L}\}$ is globally hypoelliptic on $\mathbb{T}^{n+1}$. 
\end{example}

\begin{remark}
	We point out that the operator $\mathscr{L}$ in Example \ref{log grow example} does not satisfy the H\"{o}rmander's  condition \eqref{Hormander condition}. On the other hand, the operator $\mathfrak{L}$, in Example \ref{exe-hormander-cond}, does not meet any of the logarithmic conditions in Definition \ref{growth}. 
\end{remark}

%===========================================================================
%===========================================================================
\subsubsection{Super-logarithm growth \label{sec-S-Log-cond}}
%===========================================================================
%===========================================================================

The goal of this section is to  analyze the global hypoellipticity of  system $\mathbb{L}$, in the case where there exists $j \in \mathcal{L}$ such that $L_{j0}$  is globally hypoelliptic and  $\Im\mathcal{M}_j$ satisfies either condition (ii) or condition (iii) in Defini\-tion \ref{growth}.

\begin{proposition}\label{pro-super-log}
	If some $L_{j0}$ is globally hypoelliptic on $\mathbb{T}_{t_j}^1\times\mathbb{T}_x^N$ and  either condition (ii) or condition (iii), both from  Definition \ref{growth}, is fulfilled, then  $\mathbb{L}$ is globally hypoelliptic.
\end{proposition}

\begin{proof}
We restrict the proof only to the case where condition (ii) holds, since the other one is similar to this case. Suppose, for a moment, that  $a_j(\cdot) \geq 0$ and define the following sets
	$$
	\beta_- = \{\xi \in \mathbb{Z}^N; \ \beta_j(\xi) \leq 0\}
	\ \textrm{ and } \
	\beta_+ = \{\xi \in \mathbb{Z}^N; \ \beta_j(\xi) \geq 0\}.
	$$
	
	If $\xi \in \beta_-$ is large enough, we consider $\widehat{u}(t,\xi)$ as in 
	\eqref{solutions-1}. On the other hand, for  $\xi \in \beta_+$ large enough, we use the  equivalent expression \eqref{solutions-2}.
	
	Notice that if $\xi \in \beta_-$, then it follows by \eqref{alpha-b} the existence of a constant $\gamma_1$  such that
	\begin{equation}\label{alpha-b-super-1}
	\alpha_j(\xi) \int_{t_j - \tau}^{t_j}b_j(s) ds \leq \log(\|\xi\|^{\gamma_1}).
	\end{equation}
	
	Since  $a_j(s)\beta_j(\xi) \leq 0$, we have
	\begin{align*}
	\left |
	\exp \left(-i\int_{t_j-\tau}^{t_j}\mathcal{M}_j(s,\xi)ds \right) 
	\right| 
	& \leq \|\xi\|^{\gamma_1},
	\end{align*}
	for $\|\xi\|$ large enough.
	
	Now, if $\xi \in \beta_+$, there exist a constant $\gamma_2$  such that 
	\begin{equation}\label{alpha-b-super-1}
	-\alpha_j(\xi) \int_{t_j}^{t_j+ \tau}b_j(s) ds \leq 
	\log(\|\xi\|^{\gamma_2}).
	\end{equation}
	
	In this case, $a_j(s)\beta_j(\xi) \geq 0$ and 
	\begin{align*}
	\left |
	\exp \left(i\int_{t_j}^{t_j+ \tau }\mathcal{M}_j(s,\xi)ds \right) 
	\right| 
	& \leq \|\xi\|^{\gamma_2}.
	\end{align*}

Finally, if $a_j(\cdot)\leq 0$,  the proof is completely analogous to the previous one   and can be obtained by interchanging the use of  solutions \eqref{solutions-1} and \eqref{solutions-2}.

\end{proof}

\begin{example}
	Inspired by Example 4.11 in  \cite{AGKM}, let $b(\cdot)$ be the $2\pi-$periodic extension of a real smooth nonzero function defined on $(0,2\pi)$ with integral equals to zero. Additionally, let $a(\cdot)$ be the $2\pi$-periodic extension of the function $1-\psi,$ where $\psi\in C^\infty_c((0,2\pi),\mathbb{R}),$ $0\leq \psi(t)\leq 1$ and $\psi\equiv1$ in a neighborhood of  the support of $b.$
	
	Consider a pseudo-differential operator $P(D_x)$ defined on $\mathbb{T}^1$ with symbol $p(\xi)=1+i(|\xi| \log(1 + |\xi|))$. Then
	\begin{equation*}
	\mathcal{L}=D_t +(a(t)+ib(t))P(D_x), \quad  (t,x)\in \mathbb{T}^{n+1},
	\end{equation*}
	is globally hypoelliptic and, consequently, so is  $\mathcal{L}_0$. Therefore, any system in the form $\mathbb{L}=\{L_1, \ldots, L_{n-1}, \mathcal{L}\}$  is globally hypoelliptic on $\mathbb{T}^{n+1}$.
\end{example}

%===========================================================================
%===========================================================================
\section{Reduction to normal form \label{sec-reduction}}
%===========================================================================
%===========================================================================

As remarked in the beginning of Section \ref{sec-suff-cond}, the global hypoellipticity of $\mathbb{L}_0$ is not enough  to guarantee the same for system $\mathbb{L}$. In addition, the sufficient conditions presented in Theorem \ref{main-suff} require some operator $L_{j0}$ being globally hypoelliptic, and, therefore, by Corollary \ref{coro-const-system},  system $\mathbb{L}_0$ is also globally hypoelliptic. However, it is possible for  system $\mathbb{L}$ to be globally hypoelliptic, even though the same is not true for some operator $L_{j0}$.

In fact, consider the system $\mathbb{L}$, defined on $\mathbb{T}^3$, given by
$$
L_j = D_{t_j} + a_j(t_j) P_j(D_x), \quad j=1,2, 
$$
where $\int_{0}^{2\pi}a_j(s)ds=1$ and the symbols $p_j(\xi)$ are those defined in Example \ref{example const coeff}. In this case, both $L_{j0}$ are not  globally hypoelliptic. On the other hand, since each function $\Im\mathcal{M}_j(t_j,\xi)=a_j(t)\beta_j(\xi)$ satisfies condition (iii) in Definition \ref{growth}, system $\mathbb{L}$ is globally hypoelliptic as a consequence of the results in this section (see Corollary \ref{coro-gene-psi}).

The main goal of this section is to establish conditions so that the global hypoellipticity of the associated constant system $\mathbb{L}_0$ implies the same for system $\mathbb{L}$, without imposing any further conditions on the operators $L_{j0}$.

In order to not overload our notation, we will write  
$$
\mathbb{L} \sim \mathbb{L}_0
$$
to say that system $\mathbb{L}$ is globally hypoelliptic  if and only if system  $\mathbb{L}_0$ is globally hypoelliptic.

Recalling here the notations $c_j(t_j) = a_j(t_j) + ib_j(t_j)$, $t_j\in\mathbb{T}^1$ and $p_j(\xi) = \alpha_j(\xi) + i \beta_j(\xi)$, $\xi\in\mathbb{Z}^N$, we consider the $2\pi-$periodic functions 
given by
$$
A_j(t_j)  = \int_{0}^{t_j}a_j(s)ds - a_{j0} t_j \ \textrm{ and } \
B_j(t_j)  = \int_{0}^{t_j}b_j(s)ds - b_{j0} t_j
$$
and also the functions 
$$
\mathcal{A}(t,\xi) = \sum_{j=1}^{n} p_j(\xi) A_j(t_j) \ \textrm{ and } \
\mathcal{B}(t,\xi) = \sum_{j=1}^{n} p_j(\xi) B_j(t_j).
$$

\begin{theorem}\label{t-general-reduction}
	Suppose that for each $j\in\{1,\ldots,n\}$ there exist positive cons\-tants $C_j$, $\kappa_j$ and $R_j$ satisfying
	\begin{equation}\label{general-cond-reduc}
	\sup_{\zeta \in \mathbb{T}}\left\{ 
	\exp \left(\int_{0}^{\zeta} \Im \mathcal{M}_{j}(s,\xi) ds\right) \right\} \leq 
	C_j \|\xi\|^{\kappa_j}, \quad  \|\xi\| \geq R_j.
	\end{equation}
	
	Then, 
	\begin{equation}\label{general-Psi}
	\Psi u = \sum_{\xi \in \mathbb{Z}^N} \widehat{u}(t, \xi)
	e^{\mathcal{B}(t,\xi) -i \mathcal{A}(t,\xi) } e^{i x \cdot  \xi}
	\end{equation}
	defines an isomorphism on both spaces $\mathcal{D}'(\mathbb{T}^{n+N})$ and $C^{\infty}(\mathbb{T}^{n+N})$, with inverse 
	\begin{equation*}
	\Psi^{-1} u = \sum_{\xi \in \mathbb{Z}^N} \widehat{u}(t, \xi)
	e^{-\mathcal{B}(t,\xi) + i \mathcal{A}(t,\xi) } e^{i x \cdot  \xi},
	\end{equation*}
	where  $u = \sum_{\xi \in \mathbb{Z}^N} \widehat{u}(t, \xi)  e^{i x \cdot  \xi}$. Moreover, $\mathbb{L} \sim \mathbb{L}_{0}$.
\end{theorem}

\begin{proof}
Firstly, let us prove that $\Psi$  is well defined.	By Leibniz's formula and \eqref{bound-symb}, it follows that for each $\gamma = (\gamma_1, \ldots, \gamma_n) \in \mathbb{Z}^{n}_+$, there are  positive constants $C$ and $n_0$ such that
\begin{equation}\label{exp_A}
|\partial_t^\gamma e^{- i \mathcal{A}(t,\xi)}| \leq C \|\xi\|^{\omega}  
\exp \left\{\sum_{j=1}^{n}\beta_j(\xi)A_j(t_j) \right\}
\end{equation}
and
\begin{equation}\label{exp_B}
|\partial_t^\gamma e^{\mathcal{B}(t,\xi)}| 
\leq C \|\xi\|^{\omega}  
\exp \left\{\sum_{j=1}^{n}\alpha_j(\xi)B_j(t_j) \right\},
\end{equation}
for all $\|\xi\|\geq n_0$, where $\omega = \sum_{j=1}^{n}m_j \gamma_j.$

Therefore, there are  positive constants $C$ and $n_0$ such that
\begin{align*}
\sup_{t \in \mathbb{T}^n}|\partial_t^\gamma e^{\mathcal{B}(t,\xi) -i \mathcal{A}(t,\xi) }| 
& \leq C M \|\xi\|^{\omega}  \prod_{j=1}^{n} \sup_{t_j \in \mathbb{T}_j}\left\{ 
\exp \left(\int_{0}^{t_j} \Im \mathcal{M}_{j}(s,\xi) ds\right) \right\},
\end{align*}
for all $\|\xi\|\geq n_0$, where $\omega = \sum_{j=1}^{n}m_j \gamma_j$ 
and 
$$
M = \sup_{t \in \mathbb{T}^n}\left\{ \exp \left(-\sum_{j=1}^{n}  t_j( a_{j0} + b_{j0})\right) \right\}.
$$ 

Hence, 
\begin{align*}
\sup_{t \in \mathbb{T}^n}|\partial_t^\gamma \left( \widehat{u}(t, \xi) e^{\mathcal{B}(t,\xi) -i \mathcal{A}(t,\xi) }\right)| 
& \leq C' \|\xi\|^{\omega+\kappa} 
 \sum_{\sigma = 0}^{\gamma}\binom{\gamma}{\sigma} \sup_{t \in \mathbb{T}^n}  |\partial_t^{\gamma - \sigma} \widehat{u}(t, \xi)|
\end{align*}
as $\|\xi\| \rightarrow \infty$, where $\kappa = \sum_{j=1}^{n}\kappa_j$.

Now, if $u \in \mathcal{D}'(\mathbb{T}^{n+N})$ then each term $|\partial_t^{\gamma - \sigma} \widehat{u}(t, \xi)|$ is bounded by some polynomial, implying that $\Psi u$ belongs to $\mathcal{D}'(\mathbb{T}^{n+N})$. On the other hand, if $u$ is a smooth function on  $\mathbb{T}^{n+N}$, then all the derivatives of $\widehat{u}(t, \xi)$ converge to zero faster than any polynomial and, consequently, $\Psi u$ belongs to $C^{\infty}(\mathbb{T}^{n+N})$.  A similar argumentation shows that $\Psi^{-1}$ is well defined.

Since the linearity of $\Psi$ is evident, $\mathbb{L}\sim \mathbb{L}_0$
is all that remains to be proved.  By Theorem \ref{necessary condition}, 
showing the global hypoellipticity of $\mathbb{L}_0$  would be enough to imply the same for $\mathbb{L}$.  Then, assume that $\mathbb{L}_0$ is globally hypoelliptic and consider $u \in \mathcal{D}'(\mathbb{T}^{n+N})$ to be a solution of the following  equations 
$$
L_{j} u = f_j \in C^{\infty}(\mathbb{T}^{n+N}), \quad j = 1, \ldots, n.
$$

Since $L_j = \Psi\circ L_{j0}  \circ  \Psi^{-1}$, we obtain
$$
L_{j0} (\Psi^{-1} u) = \Psi (f_j)  \in C^{\infty}(\mathbb{T}^{n+N}), \quad j = 1, \ldots, n,
$$
and, by the hypothesis on  $\mathbb{L}_0$, we get 
$\Psi^{-1} u \in  C^{\infty}(\mathbb{T}^{n+N})$ implying that $u$ is a smooth function on $\mathbb{T}^{n+N}$. Hence, $\mathbb{L}$ is globally hypoelliptic.

\end{proof}

\begin{corollary}\label{reduction-p-log}
	If $p_j(\xi)=O(\log(\|\xi\|)),$ for all $j\in\{1, \ldots n\}$, then $\mathbb{L} \sim \mathbb{L}_0$.  
\end{corollary}

\begin{proof}
	Estimate \eqref{general-cond-reduc} can be obtained by  a slight modification in the arguments used in Proposition \ref{prop_log}.
\end{proof}

In contrast with differential systems, as considered in Theorem \ref{The-diff-system}, it is possible, by Corollary \ref{reduction-p-log}, to obtain a globally hypoelliptic system, where the imaginary part of each function $c_j$ changes sign.

\begin{example}
	Let $P(D_x)$ be a pseudo-differential operator on $\mathbb{T}^1$ with symbol $p(\xi)= \log(1+ |\xi|)$ and consider the system
	$\mathbb{L}$  on $\mathbb{T}^{n+1}$ defined by
	$$
L_j = D_{t_j} + c_j(t_j)P(D_x), \quad j=1,\ldots,n.
	$$ 
	Then, $\mathbb{L}$ is globally hypoelliptic if and only if  either $b_{j0} \neq 0$ for some $j$ or the vector $(a_{10}, \ldots, a_{n0})$ is neither rational nor Liouville.	
\end{example}

The next result shows that we can consider different types of growth for the function $\Im \mathcal{M}_j$, in order to obtain the conclusion of Theorem \ref{t-general-reduction}. More precisely, since any of the conditions in Definition \ref{growth} imply  \eqref{general-cond-reduc}, we  get the following result.

\begin{corollary}\label{coro-gene-psi}
	If $\{1, \ldots, n\} \subset \mathcal{H} \cup \mathcal{L}$, then $\mathbb{L} \sim \mathbb{L}_0$.
\end{corollary}

\begin{example}\label{exe-qui-slog-hor}
	Consider   the system
	$$
L_j = D_{t_j} +  c_j(t_j) P_j(D_x), \quad j=1,2,3,
	$$
	defined on $\mathbb{T}^4$, where:
	\begin{itemize}
		\item there exists a constant $\Theta$, such that $\Im \mathcal{M}_1(\cdot, \xi) \leq \Theta$,  $\forall \xi \in \mathbb{Z}$;
		
		\item $p_2(\xi)= \log(1+ |\xi|)$;
		
		\item $p_3(\xi)= 1 +  i \xi$  and $a_3$ is positive.
		
	\end{itemize}
	
	In this case, 
	$ \mathcal{H} \cup \mathcal{L}= \{1, 2, 3\}$, then $\mathbb{L} \sim \mathbb{L}_0$.
\end{example}

%===========================================================================
%==========================================================================
\subsection{Partial  normal forms} 
%===========================================================================
%===========================================================================

This section discusses  some \textit{partial  normal forms}, namely, the systems $\mathbb{L}_{a_0}$ and  $\mathbb{L}_{b_0}$ defined respectively by
$$
L_{a_{j0}} = D_{t_j} + (a_{j0} + i b_{j}(t_j))P_j(D_x),  \quad j =1, \ldots, n
$$
and 
$$
L_{b_{j0}} = D_{t_j} + (a_{j}(t_j) + i b_{j0})P_j(D_x), \quad j =1, \ldots, n.
$$

To this end, set 
\begin{equation*}
\Psi_a u =  \sum_{\xi \in \mathbb{Z}^N} \widehat{u}(t, \xi) 
e^{-i \mathcal{A}(t,\xi)} e^{i x \cdot  \xi}
\ \textrm{ and } \
\Psi_b u =  \sum_{\xi \in \mathbb{Z}^N} \widehat{u}(t, \xi) 
e^{\mathcal{B}(t,\xi)} e^{i x \cdot  \xi}.
\end{equation*}

\begin{proposition}\label{prop-isomorphism}
	If $\beta_j(\xi)=O(\log(\|\xi\|)),$ for each $j=1, \ldots n$, then  $\Psi_a$  is an isomorphism on both the spaces $\mathcal{D}'(\mathbb{T}^{n+N})$ and $C^{\infty}(\mathbb{T}^{n+N})$. In addition, the following conjugation holds 
	\begin{equation}\label{Psi-a-1}
	L_{a_{j0}} = \Psi_a^{-1} \circ L_j \circ  \Psi_a,
	\end{equation}
	where 
	\begin{equation}
	\Psi_a^{-1} u =  \sum_{\xi \in \mathbb{Z}^N} \widehat{u}(t, \xi) 
	e^{i \mathcal{A}(t,\xi)} e^{i x \cdot  \xi}.
	\end{equation}
   In this case, $\mathbb{L} \sim \mathbb{L}_{a_0}$.

	Analogously, if $\alpha_j(\xi)=O(\log(\|\xi\|)),$ for each $j=1, \ldots n$, then $\Psi_b$  is an isomorphism on both the spaces $\mathcal{D}'(\mathbb{T}^{n+N})$ and $C^{\infty}(\mathbb{T}^{n+N}),$ and it	satisfies 
	\begin{equation}\label{Psi-b-1}
	L_{b_{j0}} = \Psi_b^{-1} \circ L_j \circ  \Psi_b,
	\end{equation}
	where 
	\begin{equation}
	\Psi_b^{-1} u =  \sum_{\xi \in \mathbb{Z}^N} \widehat{u}(t, \xi) 
	e^{-\mathcal{B}(t,\xi)} e^{i x \cdot  \xi}.
	\end{equation}
	In this case, $\mathbb{L} \sim \mathbb{L}_{b_0}$.
\end{proposition}

\begin{proof}
	Let us start with the partial normal form  $\mathbb{L}_{a_0}$. For this,  set
\begin{equation*}
\psi_{a}(t, \xi) = e^{-i \mathcal{A}(t,\xi)}\widehat{u}(t, \xi), \quad \xi \in \mathbb{Z}^N.
\end{equation*}

Given $\gamma = (\gamma_1, \ldots, \gamma_n) \in \mathbb{Z}^{N}_+$, it follows from Leibiniz's formula and the inequality \eqref{exp_A} that
\begin{equation}\label{exp-ine-A}
|\partial_t^\gamma \psi_{a}(t, \xi)| \leq C \|\xi\|^{\omega}  
\exp \left\{\sum_{j=1}^{n}\beta_j(\xi)A_j(t_j) \right\}
\sum_{\sigma = 0}^{\gamma}
\binom{\gamma}{\sigma}
|\partial_t^{\gamma - \sigma} \widehat{u}(t, \xi)|,
\end{equation}
where $\omega = \sum_{j=1}^{n}m_j \gamma_j.$  Since $\beta_j(\xi)=O(\log(\|\xi\|)),$  there exist $\kappa_j>0$ and $n_{j0}\in \mathbb{N}$ such that
\begin{equation}\label{log-estimate-3}
|\beta_j(\xi)| \leq  \log(\|\xi\|^{\kappa_j}), \quad \forall \ \|\xi\|\geq  n_{j0}.
\end{equation}

By an  approach similar to the one in the proof of Proposition \ref{prop_log}, we obtain  constants satisfying  
\begin{equation}\label{ne1}
e^{\beta_j(\xi) A_j(t_j)} \leq  \|\xi\|^{\delta_{j}}, \quad \|\xi\|\geq n_{j0},
\end{equation}

Hence, it follows from \eqref{exp-ine-A} and  \eqref{ne1} that
\begin{equation*}
|\partial_t^\gamma \psi_{a}(t, \xi)| \leq C \|\xi\|^{\omega + \delta_{j}}  
\sum_{\sigma = 0}^{\gamma}\binom{\gamma}{\sigma}
|\partial_t^{\gamma - \sigma} \widehat{u}(t, \xi)|, \quad \|\xi\| \to \infty.
\end{equation*}

Now, it is easy to see that if $u \in \mathcal{D}'(\mathbb{T}^{n+N})$, then  
$\Psi_a u \in \mathcal{D}'(\mathbb{T}^{n+N})$, while 
$u \in C^{\infty}(\mathbb{T}^{n+N})$ imply $\Psi_a u \in C^{\infty}(\mathbb{T}^{n+N}) $. By a slight modification in the above arguments, we may obtain the same conclusion for $\Psi_b$. The checking of \eqref{Psi-a-1} and \eqref{Psi-b-1} is immediate.

Now, admit that $\mathbb{L}$  is globally hypoelliptic and  let $u \in \mathcal{D}'(\mathbb{T}^{n+N})$ be a solution of
$$
L_{a_{j0}} u = f_j \in C^{\infty}(\mathbb{T}^{n+N}), \quad j = 1, \ldots, n.
$$

It follows from \eqref{Psi-a-1} that 
$$
L_j (\Psi_a u)  = \Psi_a (f_j) \in C^{\infty}(\mathbb{T}^{n+N}), \quad j = 1, \ldots, n.
$$

By the hypoellipticity of $\mathbb{L}$ we have that $\Psi_a u \in C^{\infty}(\mathbb{T}^{n+N})$, which implies $u \in C^{\infty}(\mathbb{T}^{n+N})$. Therefore,  $\mathbb{L}_{a_0}$  is globally hypoelliptic and the sufficiency is proved. The necessary part is obtained by similar arguments.  The statement $\mathbb{L} \sim \mathbb{L}_{b_0}$ can be verified by following the same ideas.

\end{proof}

\begin{remark}
Proposition \ref{prop-isomorphism} recovers the following, well known,
result for systems of  vector fields: system \eqref{system-dfifferential} is globally hypoelliptic if and only if the same holds true for the system
\begin{equation*}
L_{a_{j0}} = D_{t_j} + (a_{j0}+ ib_j(t_j)) D_x,\quad   j=1,\ldots,n.
\end{equation*}
\end{remark}

\bibliographystyle{plain}
\bibliography{references}

\end{document}